\documentclass[11pt]{article}
\usepackage[letterpaper,twoside,outer=1in,vmargin=1in,]{geometry}
\usepackage{setspace}

\usepackage{amsmath,amsfonts,amsthm,url,graphicx,subcaption,float,booktabs,bbm,enumitem,natbib}
\usepackage{algorithm,algpseudocode}
\usepackage{tikz,caption,forest}\usepackage{multicol}
\tikzset{
	treenode/.style = {shape=rectangle, rounded corners,
		draw, align=center, 
		minimum height=2ex, text depth=0.25ex,
		top color=white, bottom color=blue!20},
	root/.style     = {treenode, font=\Large\rmfamily, bottom color=red!30},
	env/.style      = {treenode, font=\ttfamily\normalsize},
}
\usetikzlibrary{arrows.meta}
\algnewcommand{\Initialize}[1]{%
	\State \textbf{Initialize:}
	\Statex \hspace*{\algorithmicindent}\parbox[t]{.8\linewidth}{\raggedright #1}
}
\algnewcommand{\algorithmicgoto}{\textbf{go to}}%
\algnewcommand{\Goto}[1]{\algorithmicgoto~step~\ref{#1}}%
\usepackage[normalem]{ulem}

\usepackage[hidelinks,bookmarksnumbered=true]{hyperref}
\usepackage{todonotes,bm}
\usepackage{tikz}
\usepackage{pgfplots}
\pgfplotsset{compat=newest}
\pgfdeclareplotmark{mystar}{
\node[star,star point ratio=2.25,minimum size=6pt,
inner sep=0pt,draw=black,solid,fill=red] {};
}
\usetikzlibrary{decorations.pathreplacing}
\usepackage{xcolor}

\usepgflibrary{arrows}
\usetikzlibrary{calc}
\usetikzlibrary{plotmarks}
\usetikzlibrary{arrows.meta,positioning}
\usetikzlibrary{patterns}

\usepackage{url}
\usepackage{longtable,tabularx,multirow,makecell}
\usepackage{amssymb}
\usepackage{psfrag}
\usepackage{amsmath}
\usepackage{setspace}
\newcommand{\exclude}[1]{}
\usepackage{bm}
\usepackage{mathtools}
\usepackage{bbm}
\usepackage[flushleft]{threeparttable}
\usepackage{algorithm}
\usepackage{algpseudocode}
\algdef{SE}[DOWHILE]{Do}{doWhile}{\algorithmicdo}[1]{\algorithmicwhile\ #1}%
\algnewcommand{\Or}{\textbf{or}}
\algnewcommand{\And}{\textbf{and}}

\usepackage{thmtools} 
\usepackage{thm-restate}

\usepackage{cleveref}
\usepackage{tikz}
\usetikzlibrary{calc}
\usetikzlibrary{plotmarks}
\usetikzlibrary{backgrounds}

\declaretheorem[name=Theorem]{theorem}
\declaretheorem[name=Proposition]{proposition}
\declaretheorem[name=Lemma]{lemma}

\declaretheorem[name=Corollary]{corollary}
\declaretheorem[name=Assumption]{assumption}

\declaretheorem[name=Example]{example}

\def\E{{\mathbb E}}

\def\Pr{{\mathbb{P}}}

\def\w{\xi}

\def\Re{\mathbb{R}}

\def\I{\mathbb{I}}

\def\hat{\widehat}
\def \w{\xi}

\def \Ze{{\mathbb{Z}}}

\def\X{{\mathcal X}}

\def\Re{{\mathbb R}}

\def\CVaR{{\mathrm{CVaR}}}
\def\VaR{{\mathrm{VaR}}}

\def\alsox{{\mathrm{ALSO}{\text-}\mathrm{X}}}

\def\alsoxs{{\mathrm{ALSO}{\text-}\mathrm{X}\#}}

\def\as{A\#}

\DeclareMathOperator{\proj}{proj}

\DeclareMathOperator*{\arginf}{arginf} 

\newcommand{\trxi}{\tilde{\bm{\xi}}}
\newcommand{\rxi}{\bm{\xi}}
\usepackage{todonotes,bm}

\usepackage{comment}

\usepackage{enumitem}
\usepackage{bigints}
\usepackage{mwe}

\newcommand{\vect}[1]{\boldsymbol{\bm{#1}}}

\renewcommand{\arraystretch}{1.5}

\newcommand*{\QEDA}{\hfill\ensuremath{\square}}
\newcommand*{\QEDB}{\hfill\ensuremath{\diamond}}

\definecolor{bluegreen}{RGB}{0,158,115}

\let\cline\cmidrule

\allowdisplaybreaks

\newcommand{\rev}[1]{#1}
\newcommand{\ack}[1]{}
\renewcommand{\citealt}[1]{\cite{#1}}
\renewcommand{\citep}[1]{\cite{#1}}
\renewcommand{\citet}[1]{\cite{#1}}

\author{Nan Jiang$^1$, Rui Chen$^2$\\
$^1$ The Hong Kong University of Science and Technology (nanjiang@ust.hk)\\
$^2$ The Chinese University of Hong Kong, Shenzhen (rchen@cuhk.edu.cn)}
\title{Tightening CVaR Approximations via Scenario-Wise Scaling for Chance-Constrained Programming}

\begin{document}
\maketitle
\begin{abstract}
    \rev{Chance-constrained programs (CCPs) provide a powerful modeling framework for decision-making under uncertainty, but their nonconvex feasible regions make them computationally challenging.} A widely used convex inner approximation replaces chance constraints with Conditional Value-at-Risk (CVaR) constraints; however, the resulting solutions can be overly conservative and suboptimal. We propose a \rev{scenario-wise} scaling approach that strengthens CVaR approximations for CCPs \rev{with finitely supported uncertainty}. \rev{The method introduces scaling factors that reweight individual scenarios within the CVaR constraint, yielding a family of potentially tighter inner approximations}. We establish sufficient conditions under which, for a suitable choice of scaling factors, the scaled CVaR approximation attains the same optimal value as the original CCP and admits a \rev{(near-)}optimal solution of the CCP. We show that these conditions are tight and further relax them in the convex setting. We also show that optimizing over \rev{scenario-wise} scaling factors is NP-hard. To address this computational challenge, we develop efficient heuristic and sequential convex approximation algorithms that iteratively update the scaling factors and generate improved feasible solutions. Numerical experiments demonstrate that the proposed methods consistently improve upon standard CVaR and state-of-the-art convex approximations, often reducing conservativeness while maintaining tractability.
\end{abstract}
\maketitle

\section{Introduction}

Chance-constrained programs (CCPs) provide a principled way to optimize decisions under uncertainty by explicitly controlling the probability of constraint violation. Specifically, a CCP chooses a decision $\bm x\in \X\subseteq\Re^n$ to minimize (without loss of generality) a linear cost subject to the requirement that some uncertain constraint holds with high probability:
\begin{align}
\label{eq_ccp}
v^* =\min_{\bm x \in \X } \left\{ \bm{c}^\top\bm x\colon\Pr\left\{\trxi \colon g(\bm x,\trxi)\leq 0\right\} \geq 1-\varepsilon\right \},
\end{align}
where $\trxi$ is a random vector, $g(\cdot,\cdot)$ models the uncertain feasibility constraint, and $\varepsilon\in(0,1)$ is a risk tolerance. Throughout this paper, we assume that CCP \eqref{eq_ccp} is feasible, and the uncertainty is finitely supported, as specified below.

\begin{assumption}
\label{assumption_finite}
The underlying probability distribution of $\trxi$ is finite, i.e. $\trxi$ has $N$ possible realizations (i.e., scenarios) $\{\rxi^1,\rxi^2,\cdots, \rxi^N\}$ with $p_i=\Pr\{\trxi=\rxi^i\}$ for $i\in[N]$, and $\sum_{i\in[N]}p_i=1$.
\end{assumption}
This assumption is \rev{common} in the CCP literature (see, e.g., \citealt{luedtke2010integer,kuccukyavuz2012mixing,qiu2014covering,ahmed2018relaxations}), \rev{as one can often approximate a general distribution using a finite sample}. With the finitely supported distribution $\Pr$, we can equivalently rewrite CCP \eqref{eq_ccp} as
\begin{align}
\label{eq_ccp_finite}
v^* =\min_{\bm x \in \X } \left\{ \bm{c}^\top\bm x\colon \sum_{i\in[N]}p_i\I\left[ g(\bm x,\rxi^i)\leq 0\right] \geq1-\varepsilon\right\},
\end{align}
where $\I[\cdot]$ is the zero-one indicator function.

Without loss of generality, we assume that the function $ g(\bm{x},\bm{\xi})$ is defined as $g(\bm{x},\bm{\xi})= \max_{j\in[J]} g_j(\bm x,\bm\xi)$, where $g_j(\bm x,\bm\xi)\colon \Re^n\times \Xi\rightarrow {\Re}$ for all $j\in[J]: =\{1,\dots,J\}$. Therefore, formulation \eqref{eq_ccp_finite} can model both single and joint chance-constrained programs. When $J = 1$, CCP~\eqref{eq_ccp_finite} includes a single chance constraint; otherwise, it contains a joint chance constraint where $g(\bm x,\rxi^i)\leq 0$ if and only if $g_j(\bm x,\rxi^i)\leq 0$ for each $j\in [J]$.  

\subsection{Motivation}\label{sec:motivation}

Despite its modeling power, CCP \eqref{eq_ccp_finite} is computationally challenging: even with linear $g(\cdot,\rxi)$ and $\X=\Re_+^n$, the problem is NP-hard \citep{luedtke2010integer}. As a result, a large body of work develops convex approximations of~\eqref{eq_ccp_finite}. A convenient way to express the chance constraint is through the violation probability.  Indeed, CCP \eqref{eq_ccp_finite} is equivalent to
\begin{align}
\label{eq_ccp_finite_1}
v^*=\min_{\bm x \in \X } \left\{ \bm{c}^\top\bm x:\; \sum_{i\in[N]}p_i\,\I\!\left[g(\bm x,\rxi^i)>0\right] \le \varepsilon\right\},
\end{align}
where we define the (nonconvex) violation probability function
\[
\bar p(\bm x):=\sum_{i\in[N]}p_i\,\I\!\left[g(\bm x,\rxi^i)>0\right].
\]
\rev{One particular approach to derive conservative inner approximations of~\eqref{eq_ccp_finite_1} is to replace $\bar p(\bm x)$ with a convex approximation $\hat p(\bm x)$ satisfying $\hat p(\bm x)\ge \bar p(\bm x)$ for all $\bm x\in\X$. Among convex approximations obtained by upper-bounding the indicator function via a convex function, the Conditional Value-at-Risk ($\CVaR$) approximation (i.e., $\hat{p}(\bm x) =  \inf_{\beta < 0}{(-\beta)^{-1}}\{\sum_{i\in[N]}p_i[g(\bm x,\rxi^i)-\beta]_+\}$) is known to be the tightest \citep{nemirovski2007convex}. Despite its tightness, the CVaR approximation differs significantly from the original CCP in many aspects. One particular property of the original CCP that is lost in the CVaR approximation is the invariance with respect to scaling of each scenario constraint.}

\rev{Specifically, for any fixed $\bm\alpha\in\Re_{++}^N$, define $h_{\bm \alpha}(\bm x,\rxi^i):=\alpha_i g(\bm x,\rxi^i)$ for $i\in[N]$. Then constraints
\[
g(\bm x,\rxi^i)\leq 0 \Leftrightarrow h_{\bm \alpha}(\bm x,\rxi^i)\leq 0
\]
are equivalent scenario by scenario. Therefore, the chance constraint itself is unchanged by positive scaling of each scenario constraint.} In contrast, convex approximations that aggregate scenario functions, such as the CVaR approximation, depend on the magnitudes of scenario constraints. Therefore, changing the scale of individual scenarios can change the strength of the approximation even though the underlying CCP is identical.
\rev{This issue is particularly visible in CCPs with a covering structure, which appear in numerous applications where certain demands have to be met with high probability (see, e.g., \citealt{deng2016decomposition,dentcheva2000concavity,shiina1999numerical,takyi1999surface,talluri2006vendor}). A common special case has $\X= \Re_+^n$, $\bm{c}\in \Re_+^n$, and $g(\bm{x},\bm{\xi})= \max_{j\in[J]} g_j(\bm x,\bm\xi)=\max_{j\in[J]}{b}_j^i-(\bm{A}_j^i)^\top\bm{x}$ with $\bm{A}_j^i\in \Re_+^{n}$, $ {b}_j^i>0$ for $ i\in [N]$ and $j\in[J]$. A standard preprocessing step \citep{qiu2014covering,xie2020bicriteria} is row-wise normalization, i.e., dividing each row ${b}_j^i-(\bm{A}_j^i)^\top\bm{x}$ by the demand ${b}_j^i$ so that each normalized constraint has a unit constant term. This normalization preserves the feasible set of the original CCP, but it can change the optimal solutions of conservative inner approximations like CVaR. More generally, different choices of the scaling vector $\bm\alpha$ in the CVaR approximation may lead to different optimal solutions and optimal objective values. A suitable choice of $\bm\alpha$ potentially reduces conservativeness and improves the optimal objective value of the approximation (see Example~\ref{example_motivate_covering} in Appendix~\ref{appendix_sec_example} for a concrete example).}

\subsection{Literature Review}

Initially introduced by \cite{charnes1958cost,charnes1963deterministic}, CCPs have been widely used to handle uncertainty in optimization problems. For example, in the energy sector, CCPs can be employed to address the challenges posed by various sources of uncertainty. Uncertainty may come from the intermittent nature of renewable energy sources, fluctuations in energy demand, potential transmission line failures, and other stochastic factors that can impact the overall performance and reliability of the energy systems (see, e.g., \citealt{cho2023exact,porras2024unifying}). In finance, CCPs are widely used in risk management, specifically through the optimization of investment portfolios under uncertain market conditions. By incorporating market volatility and other stochastic factors into financial models, CCPs enable financial analysts and portfolio managers to devise strategies that balance potential returns with acceptable levels of risk (see, e.g., \citealt{pagnoncelli2009computational,deng2021scenario}). In logistics, CCPs facilitate the effective planning and operation of supply chains by accounting for the variability in demand and uncertainties in transportation. These models allow for optimizing inventory levels, routing, and scheduling, ensuring supply chain operations remain efficient and resilient against disruptions (see, e.g., \citealt{dinh2018exact,ghosal2020distributionally}). Interested readers are referred to \cite{ahmed2008solving,kuccukyavuz2022chance} for a comprehensive review.

Despite its significance, the feasible region of CCP \eqref{eq_ccp} is generally nonconvex, making the problem challenging to solve directly. 
To address this, several approaches have been proposed in the literature to solve CCP \eqref{eq_ccp}.  One approach is the sample average approximation (SAA) method proposed by \cite{ruszczynski2002probabilistic,luedtke2008sample}, which reformulates CCP \eqref{eq_ccp} as a mixed-integer program.
Standard optimization solvers can then handle this reformulation; however, solving it to optimality may still be computationally challenging in practice.
Another approach is to develop convex inner approximations of the nonconvex chance constraint (see, e.g., \citealt{nemirovski2006scenario,calafiore2006scenario,nemirovski2007convex,ahmed2017nonanticipative}). The best-known convex approximation is to replace the chance constraint in CCP \eqref{eq_ccp} with the conditional value-at-risk ($\CVaR$) approximation proposed by \cite{nemirovski2007convex}. The CVaR approximation usually returns a feasible yet sub-optimal solution. Recently, bisection-based methods \citep{ahmed2017nonanticipative,jiang2022also,jiang2025also} like $\alsox$ and $\alsoxs$ are proposed to improve the CVaR approximation. This improvement occurs because $\alsox$ refines the solution iteratively, allowing for a more effective optimization process compared to the one-shot CVaR approximation. 

In this paper, we propose to use 
\rev{scenario-wise scaling}
to improve the CVaR approximation. Scaling the constraints is crucial for ensuring the effectiveness and efficiency of the optimization process (see, e.g., \citealt{lillo1993solving}) and is important for improving the performance of LP/MIP solvers (see, e.g., \citealt{berthold2021learning} and the references therein).
Scaling within the CVaR approximation of chance-constrained programs has appeared in the literature, particularly for achieving a better approximation for joint chance constraints (see, e.g., \citealt{chen2010cvar,zymler2013distributionally,yang2016distributionally,chen2023approximations}). \cite{chen2010cvar} use scaling to approximate joint chance-constrained problems, improving upon the standard approach using Bonferroni's inequality. \cite{zymler2013distributionally} show that CVaR approximation can be exact with scaling for distributional robust joint chance-constrained problems with moment ambiguity sets. 
\rev{However, existing works focus on scaling of individual constraints in joint chance-constrained settings and do not consider scaling applied to individual scenarios. As a result, the conditions under which the scenario-wise scaling within the CVaR approximation of a CCP is exact, and how it relates to the optimal objective value of a CCP, have remained unclear}. This work fills this gap. 
In particular, we identify sufficient conditions under which the scaled CVaR approximation preserves an optimal solution of a CCP, and \rev{we further establish NP-hardness results for the scaled CVaR approximation}.
We also develop efficient algorithms to improve CVaR approximation \rev{using scenario-wise scaling}, and our numerical study demonstrates the effectiveness of the proposed algorithms.

\noindent\textbf{Organization.}
The remainder of the paper is organized as follows: Section~\ref{sec_scaling} details the scaling procedure and illustrates the advantages of scaling. Section~\ref{sec_solution_approach} introduces an efficient heuristic. Section~\ref{sec_numerical_study} numerically validates the proposed methods. Section~\ref{sec_conclusion} concludes the paper.

\noindent\textbf{Notation.} The following notation is used throughout the paper. We use bold letters (e.g., $\vect{x},\vect{A}$) to denote vectors and matrices and use corresponding non-bold letters to denote their components. \rev{We use $\Re_+$(/$\Ze_+$) to denote the set of nonnegative real(/integer) numbers, and $\Re_{++}$ to denote the set of positive numbers, i.e., $\Re_+:=[0,\infty)$ and $\Re_{++}:=(0,\infty)$.} We let $\bm{e}$ be the vector of all ones.
Given an integer $n$, we let $[n]:=\{1,2,\ldots,n\}$. Given a real number $t$, we let $(t)_+:=\max\{t,0\}$. 
Given a finite set $I$, we let $|I|$ denote its cardinality. We let $\trxi$ denote a random vector and denote its realizations by $\bm\xi$.  

\section{\rev{Scenario-Wise Scaled} CVaR Approximation}\label{sec_scaling}
For a given random variable $\tilde{\bm{X}}$ with probability distribution $\Pr$ and cumulative distribution function $F_{\tilde{\bm{X}}}(s)=\Pr\{\tilde{\bm{X}}\leq s\}$, and a given risk parameter $\varepsilon\in(0,1)$, $(1-\varepsilon)$ Value-at-Risk ($\VaR$) of $\tilde{\bm{X}}$ is defined as $\VaR_{1-\varepsilon}(\tilde{\bm{X}}):=\min_s \{ s: F_{\tilde{\bm{X}}}(s)\geq 1-\varepsilon\}$. Following the definition of $\VaR$, $$\Pr\left\{\tilde{\bm{X}}\leq 0\right\} \geq 1-\varepsilon \Leftrightarrow \VaR_{1-\varepsilon}(\tilde{\bm X})\leq 0.$$
The corresponding Conditional Value-at-Risk ($\CVaR$) is defined by $$\CVaR_{1-\varepsilon}(\tilde{\bm{X}}):=\min_{\beta} \{ \beta+{\varepsilon}^{-1}\E_{\Pr}[ \tilde{\bm{X}}-\beta ]_+\},$$
which serves as an upper bound of $\VaR_{1-\epsilon}(\tilde{\bm X})$. Note that 
$$\min_{\beta} \{ \beta+{\varepsilon}^{-1}\E_{\Pr}[ \tilde{\bm{X}}-\beta ]_+\}\leq 0 \Leftrightarrow \min_{\beta\leq 0} \{ \beta+{\varepsilon}^{-1}\E_{\Pr}[ \tilde{\bm{X}}-\beta ]_+\}\leq 0.$$
\rev{Applying the above arguments to CCP \eqref{eq_ccp_finite}, one can then obtain a CVaR (inner) approximation \citep{nemirovski2007convex,rockafellar2000optimization} of CCP \eqref{eq_ccp_finite}}:
\begin{align}
\label{ccp_cvar_eq_1}
v^{\CVaR}=\min_{\bm x\in \X} \left\{ \bm{c}^\top\bm{x}\colon \min_{\beta\leq 0}\left\{\beta+\frac{1}{\varepsilon} \sum_{i\in[N]}p_i\left[\{ g(\bm x,\rxi^i)-\beta\}_+\right]\right\}\leq 0 \right\}.
\end{align}
Introducing nonnegative auxiliary variables $\bm s$, the CVaR approximation \eqref{ccp_cvar_eq_1} can be reformulated as
\begin{align}
\label{ccp_cvar_eq_2}
v^{\CVaR}=\min_{\bm x\in \X,\beta\leq 0,\bm s\geq \bm 0} \left\{ \bm{c}^\top\bm{x}\colon \varepsilon\beta+ \sum_{i\in[N]}p_i s_i\leq 0, s_i+\beta\geq g(\bm x,\rxi^i),~i\in[N] \right\}.
\end{align}
\rev{As discussed in Section \ref{sec:motivation},} for each constraint $g(\bm x,\rxi^i)$ with $i\in[N]$, \rev{we introduce a corresponding scaling variable $\alpha_i$ for scenario-wise scaling}.
For fixed scaling factors $\bm\alpha\geq \bm e$, we define an \rev{\emph{$\bm\alpha$-scaled}} CVaR approximation:
\begin{align}
\label{ccp_cvar_eq_scale_a}
v^{\CVaR}(\bm\alpha):=\min_{\bm x\in \X,\beta\leq 0,\bm s\geq \bm 0} \left\{ \bm{c}^\top\bm{x}\colon \varepsilon\beta+ \sum_{i\in[N]}p_i s_i\leq 0, s_i+\beta\geq \alpha_i g(\bm x,\rxi^i),~i\in[N] \right\},
\end{align}
and we denote the \rev{\emph{optimally scaled}} CVaR approximation by \begin{equation}
\label{ccp_cvar_eq_scale}
\begin{aligned}
v^{\CVaR}_S:=&\inf_{\bm\alpha\geq\bm e}v^{\CVaR}(\bm\alpha)\\
=&\inf_{\bm x\in \X,\beta\leq 0,\bm s\geq \bm 0,\bm\alpha\geq\bm e} \left\{ \bm{c}^\top\bm{x}\colon \varepsilon\beta+ \sum_{i\in[N]}p_i s_i\leq 0, s_i+\beta\geq \alpha_i g(\bm x,\rxi^i),~i\in[N] \right\}.
\end{aligned}    
\end{equation}
\rev{To the best of our knowledge, the scaled CVaR approximation formulation of the form~\eqref{ccp_cvar_eq_scale_a} is introduced in the literature for the first time. While \cite{chen2010cvar} refers to a notion of ``scaling” within the CVaR approximation, it focuses on scaling each individual constraint in the joint chance constrained setting, rather than individual scenarios. 
} In our setting, if we set $\alpha_i=1$ for each $i\in[N]$ in \eqref{ccp_cvar_eq_scale_a}, we recover the original CVaR approximation \eqref{ccp_cvar_eq_1}, \rev{i.e., $v^{\CVaR}(\bm e)=v^{\CVaR}$}. 
\rev{We note that \eqref{ccp_cvar_eq_scale_a} is precisely the CVaR approximation of a scaled version of CCP~\eqref{eq_ccp_finite}, i.e., the CVaR approximation of $v^* =\min_{\bm x \in \X } \{ \bm{c}^\top\bm x\colon \sum_{i\in[N]}p_i\I\left[ \alpha_i g(\bm x,\rxi^i)\leq 0\right] \geq1-\varepsilon\}$. It then follows that \eqref{ccp_cvar_eq_scale} is a conservative approximation of the CCP \eqref{eq_ccp_finite}, which implies the following result.}

\begin{proposition}
\label{prop_cvar_scvar}
We have \rev{$v^*\leq v^{\CVaR}(\bm\alpha)$ for all $\bm\alpha\geq \bm e$}, and therefore $v^*\leq v^{\CVaR}_S \leq v^{\CVaR}$.
\end{proposition}

Note that if $v^*= v^{\CVaR}_S$ and there exists $\hat{\bm\alpha}\geq\bm e$ such that $v^{\CVaR}_S=v^{\CVaR}(\hat{\bm\alpha})$, then any solution of \eqref{ccp_cvar_eq_scale_a} with $\bm\alpha=\hat{\bm\alpha}$ is an optimal solution of the original CCP \eqref{eq_ccp_finite} due to the conservativeness of scaled CVaR approximations.
We next identify sufficient conditions such that the optimally scaled CVaR approximation \eqref{ccp_cvar_eq_scale} provides an exact solution to CCP \eqref{eq_ccp_finite}.

\begin{theorem}[Exact case]
\label{theorem_general_case}
Let $\bm x^*$ denote an optimal solution of CCP \eqref{eq_ccp_finite} and define \rev{$I^*:=\{i\in[N]:g(\bm x^*,\rxi^i)< 0\}$}. Suppose that 
$\sum_{i\in I^*}p_i >1-\varepsilon$. 
Then, we have $v^*=v_S^{\CVaR}$, and there exists $(\hat {\bm \alpha},\hat{\beta}, \hat{\bm s})$ such that $v^{\CVaR}_S=v^{\CVaR}(\hat{\bm\alpha})$ and $(\bm x^*,\hat{\beta},\hat{\bm s})$ minimizes the scaled CVaR approximation \eqref{ccp_cvar_eq_scale}. 
\end{theorem}
\begin{proof}
\rev{Due to Proposition \ref{prop_cvar_scvar}, we only need to show that $v_S^{\CVaR} \leq v^*$.} Under the \rev{conditions} that $g(\bm x^*,\rxi^i)<0$ for all $i\in I^*$ and $\sum_{i\in I^*}p_i >1-\varepsilon$, we use $\bm x^*$ to construct a feasible solution to the scaled CVaR approximation \eqref{ccp_cvar_eq_scale}. Define $\tau:=\sum_{i\in [N]\setminus I^*}p_i\in[0,\varepsilon)$ and $ \bar{\alpha} := {\sum_{i\in [N]\setminus  I^*} p_ig(\bm x^*,\rxi^i)}/{(\varepsilon - \tau)} $. 
The construction is as follows:
\begin{align*}
& \hat{\alpha}_i = \max\left\{- \frac{\bar{\alpha}}{g(\bm x^*,\rxi^i)}, 1\right\},~i\in I^*;~\hat{\alpha}_i=1, ~i\in [N]\setminus I^*;\\
& \hat{\beta} = -\bar{\alpha};\\
& \hat{s}_i = 0,~i\in I^*;~\hat{s}_i = g(\bm x^*,\rxi^i) +\bar{\alpha},~i\in [N]\setminus I^*.
\end{align*}
Then it is easy to verify that $\hat{s}_i+\hat{\beta}\geq \hat{\alpha}_i g(\bm x^*,\rxi^i)$ for each $i\in[N]$ and 
\begin{align*}
\varepsilon\hat{\beta}+ \sum_{i\in[N]}p_i \hat{s}_i = -    \varepsilon \bar{\alpha} + \sum_{i\in [N]\setminus I^*}p_i [g(\bm x^*,\rxi^i) +\bar{\alpha}] = -(\varepsilon-\tau) \bar{\alpha} + \sum_{i\in [N]\setminus I^*} p_ig(\bm x^*,\rxi^i) \leq 0.
\end{align*}
Therefore, $(\bm x^*,\hat{\beta},\hat{\bm s},\hat{\bm\alpha})$ satisfies all constraints in the scaled CVaR approximation \eqref{ccp_cvar_eq_scale}, which implies $v^{\CVaR}(\hat{\bm\alpha})\leq \bm{c}^\top\bm{x}^*=v^*$. This completes the proof.
\QEDA
\end{proof}

We make the following remarks about \Cref{theorem_general_case}:
\begin{enumerate}[label=(\roman*)]
\item Example~\ref{example_discrete_exact} in Appendix~\ref{appendix_sec_example} illustrates the correctness of Theorem~\ref{theorem_general_case}. In fact, \rev{the two conditions} of \Cref{theorem_general_case} (i.e., $g(\bm x^*,\rxi^i)<0$ for all $i\in I^*$ and $\sum_{i\in I^*}p_i >1-\varepsilon$) are \rev{essentially tight for guaranteeing} the equivalence between the optimal objective values of the scaled CVaR approximation and its exact CCP counterpart. If either of these two \rev{conditions} in \Cref{theorem_general_case} is not satisfied, Example~\ref{example_discrete_not_exact_condition_1} and Example~\ref{example_discrete_not_exact_condition_2} in Appendix~\ref{appendix_sec_example} illustrate that the scaled CVaR approximation \eqref{ccp_cvar_eq_scale} may not provide an optimal solution;  
\item \Cref{theorem_general_case} extends to distributionally robust chance constrained programs under the type-$\infty$ Wasserstein ambiguity set, since the corresponding formulations are equivalent to CCPs of the form \eqref{eq_ccp_finite} (see, e.g., Proposition 8 in \citealt{jiang2022also}). \rev{The extension of Theorem~\ref{theorem_general_case} to type-$\infty$ Wasserstein distributionally robust chance constrained programs also suggests a conceptual connection to adversarial robustness. In particular, type-$\infty$ Wasserstein distributionally robust optimization has been shown to be closely related to adversarial training formulations in machine learning (see, e.g., \citealp{sinha2018certifiable}). Exploring whether the scaling ideas developed in this paper can be adapted to adversarial training remains an interesting direction for future research};
\item The \rev{condition} $g(\bm x^*,\rxi^i)<0$ for all $i\in I^*$ can be viewed as a special case of the setting in \cite{hanasusanto2017ambiguous}, where the authors studied a strict chance constraint that requires
strict satisfaction of the constraint $g(\bm x,\rxi^i)<0$ for all $i\in [N]$. Nevertheless, in the convex case, we will relax this \rev{condition}. In the purely discrete case, we can always slightly perturb the constraint such that, for all $x\in \X$, $g(\bm x,\rxi^i)\leq 0$ if and only if $g(\bm x,\rxi^i)<0$;
\item We remark that the second \rev{condition} is not restrictive at all. For example, in chance constraints with equiprobable scenarios (i.e., $p_i=1/N$ for each $i\in [N]$), \rev{the condition $\sum_{i\in I^*}p_i >1-\varepsilon$} can accommodate cases where $N\varepsilon$ is not an integer. Otherwise, since the distribution is finite, one can always perturb $\varepsilon$ to satisfy the \rev{condition $\sum_{i\in I^*}p_i >1-\varepsilon$ if $g(\bm x^*,\rxi^i)<0$ for all $i\in I^*$}. 
\end{enumerate}

\rev{To provide further insight into Theorem~\ref{theorem_general_case}, we consider the following example.} 
\begin{example}
\label{Example_illustration}
\rev{Let us consider the following example with $N=6$, $p_i=1/6$ for each $i\in[N]$, and $\varepsilon=5/12$. For a given optimal solution $\bm x^*$, suppose that the scenarios are ordered such that $ g(\bm x^*,\rxi^{\sigma_1})\leq g(\bm x^*,\rxi^{\sigma_2}) \leq g(\bm x^*,\rxi^{\sigma_3})\leq g(\bm x^*,\rxi^{\sigma_4})<0$ and $0<g(\bm x^*,\rxi^{\sigma_5})\leq g(\bm x^*,\rxi^{\sigma_6})$. By Theorem~\ref{theorem_general_case}, it follows that $ I^*=\{\sigma_1,\sigma_2,\sigma_3,\sigma_4\}$ and $\sum_{i\in I^*}p_i=2/3>1-\varepsilon$. This example satisfies the conditions in Theorem \ref{theorem_general_case}, which implies that $v^*=v_S^{\CVaR}$. As illustrated in Figure~\ref{fig_illustration}, in this example, $\CVaR_{1-\varepsilon}[g(\bm x^*,\rxi)]$ can be expressed as
\begin{align*}
\CVaR_{1-\varepsilon}[g(\bm x^*,\rxi)] =\frac{1}{5/12} \left[\frac{1}{12}g(\bm x^*,\rxi^{\sigma_4}) +  \frac{1}{6}g(\bm x^*,\rxi^{\sigma_5})  +  \frac{1}{6}g(\bm x^*,\rxi^{\sigma_6})\right]. 
\end{align*}
Clearly, $\sum_{i\in I^*}p_i >1-\varepsilon$ with $I^*=\{i\in[N]:g(\bm x^*,\rxi^i)< 0\}$ guarantees the scenario-wise scaled CVaR approximation to recover the optimal objective value (by choosing $\alpha_{\sigma_4}$ large enough). If we instead perturb the risk parameter to $\varepsilon=1/3$, then in this example we obtain $\sum_{i\in I^*}p_i =1-\varepsilon$. In this case, the corresponding CVaR reduces to $ \CVaR_{1-\varepsilon}[g(\bm x^*,\rxi)] =1/2[g(\bm x^*,\rxi^{\sigma_5})+g(\bm x^*,\rxi^{\sigma_6})]>0$. Even if the scenarios $\rxi^{\sigma_5},\rxi^{\sigma_6}$ are scaled with factors $\alpha_{\sigma_5}\geq 1$ and $\alpha_{\sigma_6}\geq 1$, respectively, it is impossible to satisfy
\begin{align*}
\CVaR_{1-\varepsilon}[h_{\bm\alpha}(\bm x^*,\rxi)] =1/2[\alpha_{\sigma_5}g(\bm x^*,\rxi^{\sigma_5})+\alpha_{\sigma_6}g(\bm x^*,\rxi^{\sigma_6})]\leq 0.
\end{align*}
Consequently, the scenario-wise CVaR approximation cannot recover the optimal objective value in this setting with $\sum_{i\in I^*}p_i =1-\varepsilon$.
}
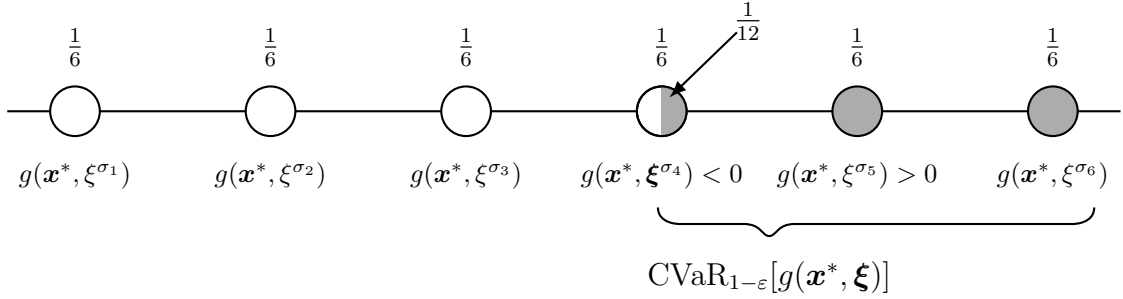
\begin{figure}
\centering
\label{fig_illustration}
\begin{tikzpicture}[
>=Latex,
thick,
every node/.style={font=\large},
glabel/.style={font=\small}
]
\def\dx{2.6}
\coordinate (p1) at (0,0);
\coordinate (p2) at (\dx,0);
\coordinate (p3) at (2*\dx,0);
\coordinate (p4) at (3*\dx,0);
\coordinate (p5) at (4*\dx,0);
\coordinate (p6) at (5*\dx,0);

\draw (-0.9,0) -- (5*\dx+0.9,0);

\newcommand{\opencirc}[1]{%
\draw[fill=white] (#1) circle (0.33);
}

\newcommand{\graycirc}[1]{%
\draw[fill=gray!65] (#1) circle (0.33);
}

\newcommand{\halfcirc}[1]{%
\draw[fill=white] (#1) circle (0.33);
\begin{scope}
\clip (#1) circle (0.33);
\path[fill=gray!65] ($(#1)+(0,-0.40)$) rectangle ($(#1)+(0.40,0.40)$);
\end{scope}
\draw (#1) circle (0.33);
}

\opencirc{p1}
\opencirc{p2}
\opencirc{p3}
\halfcirc{p4}
\graycirc{p5}
\graycirc{p6}

\node at ($(p1)+(0,0.85)$) {$\frac{1}{6}$};
\node at ($(p2)+(0,0.85)$) {$\frac{1}{6}$};
\node at ($(p3)+(0,0.85)$) {$\frac{1}{6}$};
\node at ($(p4)+(0,0.85)$) {$\frac{1}{6}$};
\node at ($(p5)+(0,0.85)$) {$\frac{1}{6}$};
\node at ($(p6)+(0,0.85)$) {$\frac{1}{6}$};

\node[glabel] at ($(p1)+(0,-0.85)$) {$g(\bm x^*,\xi^{\sigma_1})$};
\node[glabel] at ($(p2)+(0,-0.85)$) {$g(\bm x^*,\xi^{\sigma_2})$};
\node[glabel] at ($(p3)+(0,-0.85)$) {$g(\bm x^*,\xi^{\sigma_3})$};
\node[glabel] at ($(p4)+(0,-0.85)$) {$g(\bm x^*,\rxi^{\sigma_4})<0$};
\node[glabel] at ($(p5)+(0,-0.85)$) {$g(\bm x^*,\xi^{\sigma _5})>0$};
\node[glabel] at ($(p6)+(0,-0.85)$) {$g(\bm x^*,\xi^{\sigma _6})$};

\def\yB{-1.5}      
\def\endh{0.22}    
\def\w{0.34}       
\def\d{0.16}       

\coordinate (L) at ($(p4)+(-0.05,\yB)$);
\coordinate (R) at ($(p6)+(0.55,\yB)$);

\coordinate (M) at ($($(p4)!0.55!(p5)$)+(0,\yB)$);

\draw
($(L)+(0,\endh)$) .. controls ($(L)+(0,0)$) and ($(L)+(0.28,0)$) .. ($(L)+(0.28,0)$)
-- ($(M)+(-\w,0)$)
.. controls ($(M)+(-0.10,0)$) and ($(M)+(-0.10,-\d)$) .. ($(M)+(0,-\d)$)
.. controls ($(M)+(0.10,-\d)$) and ($(M)+(0.10,0)$) .. ($(M)+(\w,0)$)
-- ($(R)+(-0.28,0)$)
.. controls ($(R)+(-0.28,0)$) and ($(R)+(0,0)$) .. ($(R)+(0,\endh)$);

\node (cvar) at ($(M)+(0,-0.70)$) {$\CVaR_{1-\varepsilon}[g(\bm x^*,\rxi)]$};
\node at ($(p4)+(1.15,1.20)$) {$\frac{1}{12}$};
\draw[->] ($(p4)+(1.00,1.05)$) -- ($(p4)+(0.10,0.18)$);
\end{tikzpicture}
\caption{\rev{An Illustration of Conditions in Theorem 1 with $N=6$ and $\varepsilon=5/12$, adapted from Figure 2 of \cite{sarykalin2008value}.}}
\end{figure}
\QEDB
\end{example}
\rev{The condition $\sum_{i\in I^*}p_i >1-\varepsilon$}
in Theorem \ref{theorem_general_case} may appear restrictive. However, it can be significantly relaxed when all constraints, including the deterministic ones (i.e., constraints in $\bm x\in \X$) and the \rev{realization of the uncertain} constraint in each scenario (i.e., the function $g(\bm x,\cdot)$), are convex.

\begin{theorem}[Convex case]
\label{theorem_convex_case}
Let $\bm x^*$ denote an optimal solution of CCP \eqref{eq_ccp_finite}, and define $\bar I^*:=\{i\in[N]:g(\bm x^*,\rxi^i)\leq 0\}$ and $\bar \X^*=\{\bm x\in \X:g(\bm x,\rxi^i)\leq 0,~i\in \bar I^*\}$. Then in the convex case (i.e., $g(\cdot,\rxi^i)$ is convex for all $i\in \bar I^*$, and set $\X$ is convex), $v^*=v_S^{\CVaR}$ if \rev{there exists $\bar I\subseteq \bar I^*$ such that} (i) for each $i\in \bar I$, there exists $\bm x^i\in \bar \X^*$ such that $g(\bm x^i,\rxi^i)<0$; and (ii) $\sum_{i\in \bar I}p_i >1-\varepsilon$.
\end{theorem}
\begin{proof}
Similar to the proof of Theorem \ref{theorem_general_case}, we only need to show that $v_S^{\CVaR}\leq v^*$.
For $\epsilon\in(0,1)$, define $\bm x(\epsilon):=(1-\epsilon)\bm x^*+(\epsilon/|\bar I|)\sum_{i\in \bar  I}\bm x^i$. Due to convexity of set $\X$ and $\{g(\cdot,\rxi^i)\}_{i\in \bar I}$, we have $\bm x(\epsilon)\in \bar \X^*$ and $g(\bm x(\epsilon),\rxi^i)< 0,~i\in \bar I$. Then, following the proof of Theorem \ref{theorem_general_case} (replacing $\bm x^*$ by $\bm x(\epsilon)$), we have $v_S^{\CVaR}\leq \bm c^\top \bm x(\epsilon)$ for all $\epsilon\in (0,1)$. Therefore, $v_S^{\CVaR}\leq \lim_{\epsilon\rightarrow 0}\bm c^\top \bm x(\epsilon)=\bm c^\top\bm x^*$.
\QEDA
\end{proof}
We make the following remarks about \Cref{theorem_convex_case}:
\begin{enumerate}[label=(\roman*)]
\item Even though the infimum of \eqref{ccp_cvar_eq_scale} may not be attainable (such as Example \ref{example_motivate_covering} of Appendix~\ref{appendix_sec_example}), the scaled CVaR approximation \eqref{ccp_cvar_eq_scale} admits solutions arbitrarily close to its infimum, i.e., any $\epsilon$-optimal solution of the scaled CVaR approximation \eqref{ccp_cvar_eq_scale} implies an $\epsilon$-optimal solution of CCP \eqref{eq_ccp_finite}; 
\item Interested readers are referred to Example~\ref{example_continuous_exact} of Appendix~\ref{appendix_sec_example} for an illustration of the correctness of Theorem~\ref{theorem_convex_case}. Similarly, when the assumption $\sum_{i\in \bar{I}}p_i >1-\varepsilon$ of \Cref{theorem_convex_case} is not met, the scaled CVaR approximation \eqref{ccp_cvar_eq_scale} may fail to find the optimal solution (see, e.g., Example~\ref{example_continuous_not_exact} in Appendix~\ref{appendix_sec_example}).
\end{enumerate}
\rev{Motivated by Theorem~\ref{theorem_convex_case}, we establish a simple yet practically relevant sufficient condition under which $v^*=v_S^{\CVaR}$ holds without requiring knowledge of the optimal solution. Specifically, this sufficient condition consists of two parts, namely, (i) the existence of a point $\bar{\bm x}\in \X$ such that $g(\bar{\bm x},\rxi^i)<0$ for all $i\in[N]$, and (ii) $\sum_{i\in I}p_i\neq \varepsilon$ for all $I\subseteq[N]$. Note that part (ii) of the sufficient condition is not restrictive following remark (iv) of Theorem 1. Next, we further propose a practical procedure for identifying such a certificate $\bar{\bm x}$. Given a tolerance $\bar{\delta}< 0$ (e.g., $\bar{\delta}=-10^{-5}$), we may compute $\bar{\bm x}$ by solving the following feasibility problem:
\begin{align}
\label{condition_feasibility}
\bar{\bm x} \in \arg\min_{\bm x\in \X} \left\{ 0\colon  g(\bm x,\rxi^i)\leq \bar{\delta},~i\in[N] \right\}.
\end{align}
In a wide range of chance constrained programs, including stochastic lot-sizing problems \citep{ruszczynski2002probabilistic}, portfolio optimization \citep{xie2020bicriteria,chen2024data}, and transportation applications \citep{chen2024data}, it is easy to construct a feasible solution $\bar{\bm x}$ to problem~\eqref{condition_feasibility}.}

According to the equivalence of CCP \eqref{eq_ccp_finite} and the scaled CVaR approximation \eqref{ccp_cvar_eq_scale} in \Cref{theorem_convex_case}, and given that solving CCP \eqref{eq_ccp_finite} is in general NP-hard (see, e.g., \citealt{luedtke2010integer}), solving the scaled CVaR approximation \eqref{ccp_cvar_eq_scale} is also NP-hard. We formally present the NP-hardness result below.

\begin{lemma}[Theorem 1 in \citealt{luedtke2010integer}]\label{lem:ccp_nphard}
The optimal objective value of a chance constrained linear program is NP-hard to compute, even when the problem takes the special form $$\min_{\bm x}\left\{\sum_{i=1}^n x_i:\Pr\{\max_{i\in[n]} (\xi_i-x_i)\leq 0\}\geq 1-\varepsilon,\bm x\in \Re_+^n\right\},$$ 
and $\xi$ has a finite support with $p_1=p_2=\ldots=p_N=1/N$.
\end{lemma}

Combining Theorem \ref{theorem_convex_case} and Lemma \ref{lem:ccp_nphard}, we can show the NP-hardness of the scaled CVaR approximation \eqref{ccp_cvar_eq_scale}.
\begin{theorem}
\label{thm:scaledCVaR_nphard}
The scaled CVaR approximation \rev{\eqref{ccp_cvar_eq_scale}} is NP-hard to solve, even when $\bm c=\bm e$, $\X=\Re_+^n$, $\Xi\subset\Re^n$ and $g(\bm x,\bm \xi)=\max_{i\in[n]}(\xi_i-x_i)$.
\end{theorem}
\begin{proof}
We first show that the two assumptions in Theorem \ref{theorem_convex_case} are satisfied by the CCP 
\begin{equation}\label{Luedtke_CCP}
    \min_{\bm x}\left\{\sum_{i=1}^n x_i:\Pr\{\max_{i\in[n]} (\xi_i-x_i)\leq 0\}\geq 1-\varepsilon,\bm x\geq \bm 0\right\}.
\end{equation} 
Without loss of generality, assumption (ii) in Theorem \ref{theorem_convex_case} is satisfied by \eqref{Luedtke_CCP} because one can always add a small enough quantity to $\varepsilon$ so that \rev{the resulting CCP is equivalent to the original CCP \eqref{Luedtke_CCP}} while assumption (ii) in Theorem \ref{theorem_convex_case} is satisfied. Regarding assumption (i) in Theorem \ref{theorem_convex_case}, let $\bm\xi^{\max}_+:=\left\{(\max_{i\in [N]}\xi^i_j)_+\right\}_{j=1}^n\in \Re_+^n$. Note that $g(\bm \xi^{\max}_+ +\bm e,\bm\xi^i)<0$ for all $i\in[N]$. Therefore, assumption (i) in Theorem~\ref{theorem_convex_case} is satisfied. Consequently, Theorem~\ref{theorem_convex_case} implies that the scaled CVaR approximation of the CCP yields the optimal objective value \rev{as \eqref{Luedtke_CCP}}. Moreover, Lemma~\ref{lem:ccp_nphard} implies that the corresponding scaled CVaR approximation is NP-hard to solve.\QEDA
\end{proof}

\section{Solution Approaches}
\label{sec_solution_approach}

As we have shown in the last section, solving the scaled CVaR approximation \eqref{ccp_cvar_eq_scale} is, in general, NP-hard. In this section, we introduce some heuristics for solving the scaled CVaR approximation \eqref{ccp_cvar_eq_scale}. We first develop an efficient heuristic approach to update the scaling factors $\bm\alpha$. We then introduce a sequential convex approximation approach. 

\subsection{An Efficient Heuristic Based on Theorem~\ref{theorem_general_case}}
In contrast to the algorithms proposed in \cite{chen2010cvar,zymler2013distributionally}, our first solution approach does not optimize over $\bm\alpha$ within an optimization problem. Instead, we update  $\bm\alpha$ based on the solution construction described in \Cref{theorem_general_case}. To illustrate, we assume in this section that $\sum_{i\in I}p_i\neq \varepsilon$ for all $I\subseteq[N]$.
In Algorithm~\ref{alg_cvar_scale}, we iteratively update $\bm \alpha$ in the scaled CVaR approximation \eqref{ccp_cvar_eq_scale}. 

\vspace{-1em}
\begin{algorithm}[tp]
\caption{A Heuristic to Solve the Scaled CVaR Approximation \eqref{ccp_cvar_eq_scale}}
\label{alg_cvar_scale}
\begin{algorithmic}[1]
\State \textbf{Input:} Let $k\gets0$ with initial solution $\bm x^{(0)}$, let $\delta_1$ denote the stopping criterion parameter, let $\delta_2$ denote a violation threshold,  and initialize $\Delta \gets \infty$
\While{$\Delta \geq \delta_1$}
\State  Denote $I_k=\{i\in[N]:g(\bm x^{(k)},\rxi^i)< \delta_2\}$ and
$\bar{\alpha}=\frac{1}{\varepsilon-\sum_{i\in[N]\setminus I_k}p_i} \sum_{i\in[N]\setminus I_k} p_i g(\bm x^{(k)},\rxi^i)$. 
Update $ \bm \alpha^{(k+1)}$ as
\begin{align*}
\alpha_i^{(k+1)} = \max\left\{-\frac{ \bar{\alpha}}{g(\bm x^{(k)},\rxi^i)},1, \alpha_i^{(k)} \right\},~i\in I_k;  \alpha_i^{(k+1)}=1,~i\in [N]\setminus I_k
\end{align*}
\State  Solve the scaled CVaR approximation \eqref{ccp_cvar_eq_scale} with $ \bm \alpha^{(k+1)}$, i.e.,
\begin{align*}
& v^{\CVaR}(\bm \alpha^{(k+1)})=\min_{\bm x\in \X,\beta\leq 0,\bm s\geq \bm 0} \left\{ \bm{c}^\top\bm{x}\colon \varepsilon\beta+ \sum_{i\in[N]} p_i s_i\leq 0, s_i+\beta\geq \alpha^{(k+1)}_i g(\bm x,\rxi^i),~  i\in[N] \right\},\\
& (\bm x^{(k+1)},\beta^{(k+1)},\bm s^{(k+1)})\in\arg\min_{\bm x\in \X,\beta\leq 0,\bm s\geq \bm 0} \left\{ \bm{c}^\top\bm{x}\colon \varepsilon\beta+ \sum_{i\in[N]} p_i s_i\leq 0, s_i+\beta\geq \alpha^{(k+1)}_i g(\bm x,\rxi^i),~i\in[N] \right\}
\end{align*}
\State Let	$\Delta\gets\left\lvert \bm c^\top  \bm x^{(k)}-\bm c^\top  \bm x^{(k+1)}\right\rvert $ and $k\gets k+1$
\EndWhile
\State \textbf{Output:} A feasible solution $\bm x^{(k)}$ and its objective $\bar{v}^{\CVaR}_S=\bm c^\top  \bm x^{(k)}$
\end{algorithmic}
\end{algorithm}

\begin{proposition}
\label{proposition_heuristics}
Let $\delta_2=0$ and set the initial solution of \Cref{alg_cvar_scale} as the solution from CVaR approximation \eqref{ccp_cvar_eq_1}. The sequence of objective values $\{ \bm c^\top  \bm x^{(k)} \}_{k\in \Ze_+}$ generated by \Cref{alg_cvar_scale} is monotonically nonincreasing, bounded from below, and hence convergent.
\end{proposition}
\begin{proof}
At iteration $k+1$ of Step 4 in \Cref{alg_cvar_scale}, the scaled CVaR approximation \eqref{ccp_cvar_eq_scale} we are solving is
\begin{align}
\label{ccp_cvar_eq_scale_alpha}
v^{\CVaR}(\bm \alpha^{(k+1)})=\min_{\substack{\bm x\in \X,\\\beta\leq 0,\bm s\geq \bm 0}} \left\{ \bm{c}^\top\bm{x}\colon \varepsilon\beta+ \sum_{i\in[N]} p_i s_i\leq 0, s_i+\beta\geq \alpha^{(k+1)}_i g(\bm x,\rxi^i),~i\in[N] \right\}.
\end{align}
We only need to show that $\bm x^{(k)}$ is feasible to the scaled CVaR approximation \eqref{ccp_cvar_eq_scale_alpha}. Let $(\bm x^{(k)},\beta^{(k)},\bm s^{(k)})$ be an optimal solution at the $k$-th iteration after solving the scaled CVaR approximation with $\bm\alpha^{(k)}$. Then we have
\begin{align*}
& s^{(k)}_i+\beta^{(k)}\geq \alpha^{(k)}_i g(\bm x^{(k)},\rxi^i) \geq g(\bm x^{(k)},\rxi^i) = \alpha^{(k+1)}_i g(\bm x^{(k)},\rxi^i),~i\in [N]\setminus I_k,\\
& s^{(k)}_i+\beta^{(k)} \geq \alpha^{(k)}_i g(\bm x^{(k)},\rxi^i) \geq \alpha^{(k+1)}_i g(\bm x^{(k)},\rxi^i),~i\in I_k.
\end{align*}
Hence, $\bm x^{(k)}$ is feasible to the scaled CVaR approximation \eqref{ccp_cvar_eq_scale_alpha}. Thus, the sequence of $\{ \bm c^\top  \bm x^{(k)} \}_{k\in \Ze_+}$ is monotonically nonincreasing. Given that the objective function $\bm c^\top \bm x$ is bounded below by $v^*$, the monotonicity of the sequence of $\{ \bm c^\top  \bm x^{(k)} \}_{k\in \Ze_+}$ implies its convergence.
\QEDA
\end{proof}

We make the following remarks about \Cref{proposition_heuristics}:
\begin{enumerate}[label=(\roman*)]
\item  From \Cref{proposition_heuristics},  we can demonstrate that\rev{, using the CVaR solution as an initial solution,} the output of \Cref{alg_cvar_scale} is no worse than the CVaR approximation \eqref{ccp_cvar_eq_1}, i.e., $\bar{v}^{\CVaR}_S\leq {v}^{\CVaR}$. 
When the CVaR approximation \eqref{ccp_cvar_eq_1} is infeasible (see, e.g., Example~\ref{example_discrete_exact} of Appendix~\ref{appendix_sec_example}), other solutions, such as ones generated by the $\alsoxs$ approach proposed by \citet{jiang2025also}, can be used as an initial solution;
\item The scaling coefficient updating procedure proposed in \Cref{alg_cvar_scale} can be easily incorporated into other $\CVaR$-based approximation algorithms. For example, the scaling coefficient updating procedure can be integrated into $\alsoxs$ (see, e.g., \citealt{jiang2025also}) to enhance its performance further. For detailed discussions, we refer interested readers to Appendix~\ref{sec_scaling_alsoxs};
\item At each iteration, we can warm-start the process with the solution found in the previous iteration;
\item In the implementation of our numerical experiments, to avoid an excessively large $\bar{\bm\alpha}$ \rev{(which may lead to numerical issues)}, in Step 3 of \Cref{alg_cvar_scale}, we select $\delta_2=-0.005$. Note that, in such cases, the sequence of objective values $\{ \bm c^\top  \bm x^{(k)} \}_{k\in \mathbb{Z}_+}$ generated may not converge, and we may set a maximum iteration limit. If the maximum iteration limit is reached, the minimum value from the sequence $\{ \bm c^\top  \bm x^{(k)} \}_{k\in \mathbb{Z}_+}$ is reported as the final output.
\end{enumerate}
In the subsequent subsection, we discuss how to use sequential convex approximations to obtain a \rev{stronger} initial solution for \Cref{alg_cvar_scale}.

\subsection{A Sequential Convex Approximation Approach}


\rev{Recall that throughout the paper we define $ g(\bm{x},\bm{\xi})= \max_{j\in[J]} g_j(\bm x,\bm\xi)$, which allows formulation \eqref{eq_ccp_finite} to cover both single ($J=1$) and joint ($J>1$) chance constrained programs. Using this definition,} we can equivalently write the \rev{optimally} scaled CVaR approximation \eqref{ccp_cvar_eq_scale} as
\begin{align*}
v^{\CVaR}_S=\inf_{\substack{\bm x\in \X,\beta\leq 0,\\ \bm s\geq \bm 0,\bm\alpha\geq \bm e}} \left\{ \bm{c}^\top\bm{x}\colon \varepsilon\beta+ \sum_{i\in[N]}p_i s_i\leq 0, s_i+\beta\geq \alpha_i g_j(\bm x,\rxi^i),~i\in[N],j\in[J] \right\}.
\end{align*}
The product terms $\{\alpha_i g_j(\bm x,\rxi^i)\}_{i\in[N],j\in[J]}$ bring difficulty to the solution of the scaled CVaR approximation. One way to address it is to use the well-known difference-of-convex (DC) approach (see, e.g., Section 2 of \citealt{tao1997convex}). Using the fact that $xy=1/4[(x+y)^2-(x-y)^2]$, we can reformulate the scaled CVaR approximation \eqref{ccp_cvar_eq_scale} as
\begin{align*}
v^{\CVaR}_S=\inf_{\substack{\bm x\in \X,\beta\leq 0, \\ \bm s\geq \bm 0,\bm\alpha\geq \bm e}} \left\{ \bm{c}^\top\bm{x}\colon \begin{array}{l}
\displaystyle \varepsilon\beta+ \sum_{i\in[N]}p_i s_i\leq 0, \\
\displaystyle 4s_i+4\beta\geq [\alpha_i +g_j(\bm x,\rxi^i)]^2-[\alpha_i -g_j(\bm x,\rxi^i)]^2,~i\in[N],j\in[J]
\end{array}\right\}.
\end{align*}
This reformulation allows us to apply sequential convex conservative approximations. To ensure the validity of the sequential convex approximations, throughout this subsection, we assume that the function $ g_j(\bm{x},\rxi)$ is affine in $\bm x$ for each $j\in[J]$, i.e., $ g_j(\bm{x},\rxi^i) = {\bm a}_j(\xi^i)^\top \bm x- b_j(\xi^i)$  with ${\bm a}_j:\Xi\rightarrow \Re^n$, and $b_j:\Xi\rightarrow\Re$. \rev{This affine structure arises in a broad class of CCP models and has been widely used in the CCP literature (see, e.g., \citealt{qiu2014covering,ahmed2018relaxations,xie2020bicriteria,xie2021distributionally,ho2022distributionally,ho2023strong,chen2024data}).}

In the sequential convex approximation, at iteration $k+1$, for each $i\in[N],j\in[J]$, we replace $[\alpha_i -g_j(\bm x,\rxi^i)]^2$ by its first-order Taylor approximation using the solution from the previous iteration $k$, that is,
\begin{align*}
[\alpha_i -g_j(\bm x,\rxi^i)]^2 & \geq   [\alpha^{(k)}_i -g_j(\bm x^{(k)},\rxi^i)]^2 + 2[\alpha^{(k)}_i -g_j(\bm x^{(k)},\rxi^i)](\alpha_i-\alpha^{(k)}_i) \\
& \quad + 2[\alpha^{(k)}_i -g_j(\bm x^{(k)},\rxi^i)]
{\nabla_{\bm x^{(k)}} g_j(\bm x,\rxi^i)}^\top[\bm x-\bm x^{(k)}],
\end{align*}
where ${\nabla_{\bm x} g_j(\bm x,\rxi^i)}$ denotes the derivative of $g_j(\bm x,\rxi^i)$ for $\bm x^{(k)}$. 
Then, we solve the following program:
\begin{equation}\label{scale_cvar_dc}
\begin{aligned}
(\bm x^{(k+1)},\bm \alpha^{(k+1)}) \in 
& \proj_{(\bm{x},\bm{\alpha})} 
\arg\min_{\substack{\bm x\in \X, \, \beta\leq 0, \\ \bm s\geq \bm 0, \, \bm\alpha\geq \bm e}}~ \bm{c}^\top\bm{x}, \\
\textup{s.t.} \quad 
& \varepsilon\beta + \sum_{i\in[N]} p_i s_i \leq 0, \\
& [\alpha_i + g_j(\bm x,\rxi^i)]^2 \leq 4[s_i+\beta] 
+ [\alpha^{(k)}_i - g_j(\bm x^{(k)},\rxi^i)]^2 \\
& \quad 
+ 2[\alpha^{(k)}_i - g_j(\bm x^{(k)},\rxi^i)](\alpha_i-\alpha^{(k)}_i) \\
& \quad 
+ 2[\alpha^{(k)}_i - g_j(\bm x^{(k)},\rxi^i)] 
{\nabla_{\bm x^{(k)}} g_j(\bm x,\rxi^i)}^\top[\bm x-\bm x^{(k)}], ~i\in[N],j\in[J].
\end{aligned}
\end{equation}
Problem \eqref{scale_cvar_dc} serves as a convex approximation of scaled CVaR approximation \eqref{ccp_cvar_eq_scale}\rev{, which is iteratively solved in a sequential convex approximation algorithm as summarized in \Cref{alg_cvar_scale_dc}}. 
\begin{algorithm}[tp!]
\caption{A Sequential Convex Approximation Algorithm to Solve the Scaled CVaR Approximation \eqref{ccp_cvar_eq_scale}}
\label{alg_cvar_scale_dc}
\begin{algorithmic}[1]
\State \textbf{Input:} Let $k\gets 0$ with initial solution $\bm x^{(0)}$, let $\delta_1$ denote the stopping tolerance parameter, 
and initialize $\Delta \gets \infty$
\While{$\Delta \geq \delta_1$}
\State  Solve Problem \eqref{scale_cvar_dc} to obtain $\bm x^{(k+1)}$
\State Let	$\Delta\gets\left\lvert \bm c^\top  \bm x^{(k)}-\bm c^\top  \bm x^{(k+1)}\right\rvert $ and $k\gets k+1$
\EndWhile
\State \textbf{Output:} A feasible solution $\bm x^{(k)}$ and its objective value $\bar{v}^{\CVaR}_S(DC)=\bm c^\top\bm x^{(k)}$
\end{algorithmic}
\end{algorithm}

Algorithm \ref{alg_cvar_scale_dc} is similar to the sequential convex approximation (SCA) algorithm proposed by \cite{hong2011sequential}. Based on the property 1 and property 2 in \cite{hong2011sequential}, by initializing \Cref{alg_cvar_scale_dc}
with the solution obtained from the CVaR approximation \eqref{ccp_cvar_eq_1}, the sequence of objective values $\{ \bm c^\top  \bm x^{(k)} \}_{k\in \Ze_+}$ generated by \Cref{alg_cvar_scale_dc} is monotonically nonincreasing, bounded from below, and therefore convergent.
Hence, we can show that the output of \Cref{alg_cvar_scale_dc} improves upon the CVaR approximation \eqref{ccp_cvar_eq_1}, i.e., $\bar{v}^{\CVaR}_S(DC)\leq {v}^{\CVaR}$, \rev{if Algorithm \ref{alg_cvar_scale_dc} is initialized using the CVaR solution}. These properties of  \Cref{alg_cvar_scale_dc} are formally summarized below.

\begin{corollary}
\label{proposition_sequential}
Suppose we set the initial solution of \Cref{alg_cvar_scale_dc} as the solution from CVaR approximation \eqref{ccp_cvar_eq_1}, then the sequence of objective values $\{ \bm c^\top  \bm x^{(k)} \}_{k\in \Ze_+}$ generated by \Cref{alg_cvar_scale_dc} is monotonically nonincreasing, bounded from below, and hence convergent, and the output of \Cref{alg_cvar_scale_dc} is better than CVaR approximation \eqref{ccp_cvar_eq_1}, i.e., $\bar{v}^{\CVaR}_S(DC)\leq {v}^{\CVaR}$. 
\end{corollary}
We remark that even though Algorithm \ref{alg_cvar_scale_dc} remains a difference-of-convex algorithm (DCA), it is different from the DCA proposed in \citet{hong2011sequential}, where the latter directly approximates the chance constraint using difference-of-convex functions. 

Different from \Cref{alg_cvar_scale}, which updates $\bm\alpha$ and $\bm x$ separately, \Cref{alg_cvar_scale_dc} updates them simultaneously. However, note that \Cref{alg_cvar_scale_dc} may encounter numerical issues when using commercial solvers because Problem~\eqref{scale_cvar_dc} involves many quadratic constraints. To reduce the number of quadratic constraints in Problem~\eqref{scale_cvar_dc}, we take a hybrid strategy inspired by Algorithm~\ref{alg_cvar_scale}. Specifically, let $I_k=\{i\in[N]:g(\bm x^{(k)},\rxi^i)< 0\}$. Rather than relaxing all product terms $\{\alpha_i g_j(\bm x,\rxi^i)\}_{i\in[N],j\in[J]}$, the key idea is to relax only those terms where $i\in I_k$. For $i\in [N]\setminus I_k$, we simply set $\alpha_i=1$. The details of this hybrid approach are presented in \Cref{alg_cvar_scale_hybrid}.

Moreover, combining elements of \Cref{alg_cvar_scale_dc} and \Cref{alg_cvar_scale} can further enhance the solution quality. The key observation is that, when Algorithm~\ref{alg_cvar_scale_dc} terminates with the prespecified tolerance or encounters numerical difficulties, its output provides a feasible solution $\bm x^{(k)}$ to the scaled CVaR approximation with the scaling vector fixed at its final value $\bm\alpha^{(k)}$. Therefore, after Algorithm~\ref{alg_cvar_scale_dc} terminates, we can fix the scaling vector at the final iterate $\bm\alpha^{(k)}$ and solve Step 4 of Algorithm~\ref{alg_cvar_scale} once more. This step cannot worsen the objective value returned by Algorithm~\ref{alg_cvar_scale_dc}.


\begin{corollary}
\label{proposition_hybrid} 
Let $(\bm x^{(k)},\bm\alpha^{(k)})$  be the output of \Cref{alg_cvar_scale_dc}. Then solving Step 4 of Algorithm~\ref{alg_cvar_scale} with $\bm \alpha = \bm\alpha^{(k)}$ cannot produce an objective value worse than that returned by \Cref{alg_cvar_scale_dc}. In particular, the optimal value of the scaled CVaR approximation with alpha fixed at $\bm \alpha^{(k)}$ is no greater than the objective value produced by \Cref{alg_cvar_scale_dc}.
\end{corollary}

Motivated by \Cref{proposition_hybrid}, we can first apply \Cref{alg_cvar_scale_dc} to generate an initial solution, followed by \Cref{alg_cvar_scale} to refine subsequent iterations. Because both algorithms preserve monotonicity in the objective sequence, the final output sequence is consistently nonincreasing. A summary of this hybrid approach is provided in \Cref{alg_cvar_scale_hybrid}, and its detailed implementation is discussed in Section~\ref{sec_numerical_study}.

\begin{algorithm}[tp!]
\caption{A Hybrid Algorithm to Solve the Scaled CVaR Approximation \eqref{ccp_cvar_eq_scale}}
\label{alg_cvar_scale_hybrid}
\begin{algorithmic}[1]
\State \textbf{Input:} Let $k\gets 0$ with initial solution $\bm x^{(0)}$, let $\delta_1$ denote the stopping tolerance parameter, let $\delta_2$ denote a violation threshold, and initialize $\Delta \gets \infty$
\While{$\Delta \geq \delta_1$}
\State   Denote $I_k=\{i\in[N]:g(\bm x^{(k)},\rxi^i)< \delta_2\}$ and $\bar{I}_k=[N]\setminus I_k$
\State Obtain $(\bm x^{(k+1)},\bm \alpha^{(k+1)})$ by solving
\begin{align}
(\bm x^{(k+1)},\bm \alpha^{(k+1)})\in &\proj_{(\bm{x},\bm{\alpha})}\arginf_{\substack{\bm x\in \X,\beta\leq 0,\\\bm s\geq \bm 0,\bm\alpha\geq \bm e}}   \bm{c}^\top\bm{x},\nonumber\\
\textup{s.t.} \quad & \varepsilon\beta+ \sum_{i\in[N]} p_i s_i\leq 0, \nonumber\\
& s_i+\beta\geq g_j(\bm x,\rxi^i),  \alpha_i=1,~i\in \bar{I}_k, j\in[J]\nonumber\\
& [\alpha_i +g_j(\bm x,\rxi^i)]^2 \leq 4[s_i+\beta] + [\alpha^{(k)}_i -g_j(\bm x^{(k)},\rxi^i)]^2 \nonumber \\
& \quad\quad\quad\quad\quad\quad\quad + 2[\alpha^{(k)}_i -g_j(\bm x^{(k)},\rxi^i)](\alpha_i-\alpha^{(k)}_i) \nonumber \\
& \quad\quad\quad\quad\quad\quad\quad + 2[\alpha^{(k)}_i -g_j(\bm x^{(k)},\rxi^i)]
{\nabla_{\bm x^{(k)}}g_j(\bm x,\rxi^i)}^\top[\bm x-\bm x^{(k)}],~i\in I_k,j\in[J]\nonumber
\end{align}
\State Let	$\Delta\gets\left\lvert \bm c^\top  \bm x^{(k)}-\bm c^\top  \bm x^{(k+1)}\right\rvert $ and $k\gets k+1$
\EndWhile
\State Let $(\bm x^{(k+1)},\bm \alpha^{(k+1)})$ be the input of \Cref{alg_cvar_scale}, run \Cref{alg_cvar_scale} until it terminates, and let $\bm x^H$ and its objective value $\bar{v}^{\CVaR}_S(H)$ 
be the output of \Cref{alg_cvar_scale}
\State \textbf{Output:} A feasible solution $\bm x^H$ and its objective value $\bar{v}^{\CVaR}_S(H)=\bm c^\top \bm x^H$
\end{algorithmic}
\end{algorithm}

\subsection{Incorporating Auxiliary Optimality Information}

Based on the solution construction in \Cref{theorem_convex_case} and \Cref{alg_cvar_scale}, we observe that if a solution violates a scenario, the corresponding constraint will not be scaled up. This observation motivates us to incorporate auxiliary optimality information to fix the value of the corresponding $\alpha$ in the scaled CVaR approximation \eqref{ccp_cvar_eq_scale}. We employ a straightforward and effective method to incorporate the optimality information and fix $\bm\alpha$.
\vspace{-1em}
\begin{corollary}
\label{corollary_auxiliary_optimality}
Let $\bm x^*$ denote an optimal solution of CCP \eqref{eq_ccp_finite} and define $I^*:=\{i\in[N]:g(\bm x^*,\rxi^i)\leq 0\}$.
Under the assumptions of Theorem \ref{theorem_general_case} or Theorem \ref{theorem_convex_case}, 
we have
\begin{align*}
v^*=\inf\{v_S^{\CVaR}(\bm\alpha):\alpha_i=1,~i\in [N]\setminus I^*,\alpha_i\geq 1,~i\in I^*\}.
\end{align*} 
\end{corollary}
In the next section, we demonstrate how Corollary~\ref{corollary_auxiliary_optimality} can be integrated into our algorithmic implementations.

\section{Numerical Study}
\label{sec_numerical_study}
In this section, we numerically demonstrate the effectiveness of the proposed methods. All the instances in this section are executed in Python 3.9 with calls to solver Gurobi (version 11.0.3 with default settings) on a personal PC with an Apple M1 Pro processor and 16GB of memory. To evaluate the effectiveness of the proposed algorithms, we use ``\textrm{Improvement}'' to denote the percentage of differences between the value of a proposed algorithm and CVaR approximation, i.e.,
\begin{align*}
\textrm{Improvement }(\%) = \frac{ \CVaR \textrm{ approximation value}-\textrm{value of a proposed algorithm} }{|\CVaR \textrm{ approximation value}|}\times 100\%.
\end{align*}

\subsection{Joint Chance Constraint: \rev{Instances from \cite{song2014chance}}}
\label{sec_numerical_study_joint}
The joint CCP that we test admits the following form:
\begin{align*}
v^* = \min_{\bm x\in[0,1]^n}\left\{ \bm c^\top \bm x\colon \frac{1}{N} \sum_{i\in[N]}\I\left[ \max_{j\in[J]}\left\{\sum_{k\in[n]}\xi^i_{j,k} x_k - b^i_j\right\} \leq 0 \right]\geq 1-\varepsilon\right\}.
\end{align*}
We evaluate the proposed method on three sets of joint CCP instances with $N=3000$, {\textit{1-4-multi-3000}}, {\textit{1-6-multi-3000}}, and {\textit{1-7-multi-3000}} from \cite{song2014chance}. To approximate the scaled CVaR approximation \eqref{ccp_cvar_eq_scale}, we employ the following two steps:

\noindent\textbf{Step 1.} We use the  solution of CVaR approximation \eqref{ccp_cvar_eq_1} as the initial starting point for Algorithm~\ref{alg_cvar_scale_hybrid} and run Algorithm~\ref{alg_cvar_scale_hybrid} for one iteration to determine the scaling factor $\bm\alpha$.

\noindent\textbf{Step 2.} We use the scaling factor $\bm\alpha$ from Step 1 as input to Algorithm~\ref{alg_cvar_scale} and  implement Algorithm~\ref{alg_cvar_scale} with $\delta_2=-0.005$ for $25$ iterations to obtain the best objective value. 

For both Step~1 and Step~2, we consider incorporating the auxiliary optimality information discussed in Corollary~\ref{corollary_auxiliary_optimality}. We use $v^*_{U}$ to denote the best upper bound of the scaled CVaR approximation \eqref{ccp_cvar_eq_scale} 
and solve $\eta_{i}=\min_{\bm x\in[0,1]^n}\{\bm c^{\top} \bm x: \sum_{k\in[n]}\xi^i_{j,k} x_k - b^i_j \leq 0,~j\in[J]\}$ for each $i\in [N]$. Following the approach recently employed in \cite{jiang2024terminator}, if $\eta_{i} > v^*_{U}$ for some $i\in [N]$, we set $\alpha_i=1$.

We record the total running time of the steps above as the running time of approximating the scaled CVaR approximation \eqref{ccp_cvar_eq_scale}. For comparison, we also implement the $\alsox\#$ algorithm from \cite{jiang2025also}. In $\alsox\#$ (see the detailed algorithm in Algorithm~\ref{alg_alsox_sharp} of Appendix~\ref{sec_scaling_alsoxs}), we choose $t_L$ as the quantile bound from \citet{song2014chance}, $t_U$ as the CVaR approximation \eqref{ccp_cvar_eq_1}, and $\delta_A=0.05$. 
To ensure that $N\varepsilon$ is not an integer, we consider the risk parameters \rev{$\varepsilon\in\{0.050333,0.100333,0.200333,0.300333\}$}. \rev{Table~\ref{tab_joint_eps_020_summary} summarizes the average numerical performance across instances, while the detailed results are reported in Table~\ref{tab_joint_eps_020} in Appendix~\ref{appendix_numerical_study}}. 

We observe that our approach to solving the scaled CVaR approximation \eqref{ccp_cvar_eq_scale} consistently improves the objective value compared to the CVaR approximation \eqref{ccp_cvar_eq_1}. In most cases, 
our method achieves greater improvement than the $\alsox\#$ method and demonstrates superior average performance. 
Note that our approach is different from $\alsox\#$ and does not rely on the bisection procedure that the latter requires. 
Although our method for approximating the scaled CVaR approximation \eqref{ccp_cvar_eq_scale} requires a longer computation time than the $\alsox\#$ method, due to the quadratic constraints in Algorithm~\ref{alg_cvar_scale_hybrid} and the increased number of iterations in Algorithm~\ref{alg_cvar_scale}, it yields better objective value improvements, highlighting its effectiveness despite the added computational cost. We also remark that we compare our method with the approach that uses an interior point local solver IPOPT \citep{wachter2006implementation} to directly solve the scaled CVaR approximation \eqref{ccp_cvar_eq_scale}.
\rev{The fact that our method outperformed IPOPT further suggests that the nonconvex optimization problem \eqref{ccp_cvar_eq_scale} is challenging to solve in practice.}
Detailed numerical comparisons can be found in Table~\ref{tab_joint_ipopt_eps_020} 
of Appendix~\ref{appendix_sec_numerical_ipopt}.

\begin{table}[tb!]
\vspace{-1em}
\centering
\scriptsize
\caption{Numerical Results of a Joint CCP with \rev{Different} $\varepsilon$}
\renewcommand{\arraystretch}{1}
\label{tab_joint_eps_020_summary}
\setlength{\tabcolsep}{2pt}
\rev{
\begin{tabular}{ccccrrrrrrr}
\hline
\multirow{2}{*}{$\varepsilon$} & \multirow{2}{*}{$n$} & \multirow{2}{*}{$J$} & \multirow{2}{*}{Instance} & \multirow{2}{*}{\makecell{CVaR \\ Approximation}} & \multicolumn{3}{c}{$\alsox\#$} & \multicolumn{3}{c}{Scaled CVaR Approximation} \\
\cline{6-11}
&  &  &  &  & Value & Time (s) & Improvement & Value & Time (s) & Improvement \\
\hline
\multirow{3}{*}{$0.050333$}
& $20$ & $10$ & 1-4-multi-3000 & $-5778.50$ & $-5845.66$ & $83.83$ & $1.16\%$ & $-5851.09$ & $525.05$ & $1.26\%$ \\
& $39$ & $5$  & 1-6-multi-3000 & $-9864.42$ & $-9976.61$ & $83.86$ & $1.14\%$ & $-9979.76$ & $530.35$ & $1.17\%$ \\
& $50$ & $5$  & 1-7-multi-3000 & $-15766.16$ & $-15879.72$ & $108.17$ & $0.72\%$ & $-15885.74$ & $690.25$ & $0.76\%$ \\
\hline
\hline
\multirow{3}{*}{$0.100333$}
& $20$ & $10$ & 1-4-multi-3000 & $-5826.18$ & $-5901.36$ & $84.28$ & $1.29\%$ & $-5905.43$ & $531.48$ & $1.36\%$ \\
& $39$ & $5$  & 1-6-multi-3000 & $-9947.12$ & $-10081.34$ & $84.15$ & $1.35\%$ & $-10086.58$ & $538.40$ & $1.40\%$ \\
& $50$ & $5$  & 1-7-multi-3000 & $-15852.28$ & $-15995.49$ & $108.00$ & $0.90\%$ & $-16004.91$ & $691.10$ & $0.96\%$ \\
\hline
\hline
\multirow{3}{*}{$0.200333$}
& $20$ & $10$ & 1-4-multi-3000 & $-5880.51$ & $-5970.65$ & $99.69$ & $1.53\%$ & $-5977.62$ & $660.87$ & $1.65\%$ \\
& $39$ & $5$  & 1-6-multi-3000 & $-10047.76$ & $-10219.45$ & $85.62$ & $1.71\%$ & $-10226.30$ & $543.84$ & $1.78\%$ \\
& $50$ & $5$  & 1-7-multi-3000 & $-15962.55$ & $-16145.73$ & $107.37$ & $1.15\%$ & $-16153.17$ & $688.10$ & $1.19\%$ \\
\hline
\hline
\multirow{3}{*}{$0.300333$}
& $20$ & $10$ & 1-4-multi-3000 & $-5917.73$ & $-6022.59$ & $83.89$ & $1.77\%$ & $-6031.27$ & $563.86$ & $1.92\%$ \\
& $39$ & $5$  & 1-6-multi-3000 & $-10120.96$ & $-10325.93$ & $85.49$ & $2.03\%$ & $-10332.91$ & $541.71$ & $2.09\%$ \\
& $50$ & $5$  & 1-7-multi-3000 & $-16041.49$ & $-16260.96$ & $108.01$ & $1.37\%$ & $-16268.84$ & $686.50$ & $1.42\%$ \\
\hline
\multicolumn{7}{r}{Average} & $\mathbf{1.34}\%$ &  &  & $\mathbf{1.41}\%$ \\
\hline
\end{tabular}
}
\end{table}

\subsection{\rev{Portfolio Optimization}}
\rev{In this subsection, we study a portfolio optimization problem adapted from \cite{xie2020bicriteria}, with relevant formulations appearing in \cite{qiu2014covering,chen2024data,ho2023strong}. We consider the following formulation, where the decision vector is restricted to $[0,1]^n$ and the total allocation is controlled by an investment budget constraint $\sum_{k\in[n]}x_k\leq 0.2n$:
\begin{align*}
v^* = \min_{\bm x\in[0,1]^n}\left\{ \bm c^\top \bm x\colon \frac{1}{N} \sum_{i\in[N]}\I\left[ \sum_{k\in[n]}\xi^i_k x_k \geq 1 \right]\geq 1-\varepsilon, \sum_{k\in[n]}x_k\leq 0.2n\right\}.
\end{align*}
For numerical experiments, we follow the data generation scheme in \cite{xie2020bicriteria}, where $n$ is fixed to $50$, $N \in \{500,1000\}$, $\{\rxi^i\}_{i\in[N]}$ are i.i.d. uniform random variables in the range from $0.8$ to $1.2$ and the risk parameter $\varepsilon \in \{0.050333, 0.100333, 0.200333, 0.300333\}$. For each combination of $(n,N,\epsilon)$, we generate 5 instances. The cost vector $\bm c$ is generated at random, where each component takes an integer value uniformly sampled from $\{1,\cdots,100\}$. We employ the following two steps to approximate the scaled CVaR approximation \eqref{ccp_cvar_eq_scale}:}

\rev{
\noindent\textbf{Step 1.} We use the  solution of CVaR approximation \eqref{ccp_cvar_eq_1} as the initial starting point for Algorithm~\ref{alg_cvar_scale_hybrid}.
}

\rev{
\noindent\textbf{Step 2.} We implement Algorithm~\ref{alg_cvar_scale} with $\delta_2=-0.005$ for $20$ iterations to obtain the best objective value. 
}

\rev{
Similar to the joint chance constraint case discussed in Section~\ref{sec_numerical_study_joint}, we incorporate the auxiliary optimality information outlined in Corollary~\ref{corollary_auxiliary_optimality} in both Step~1 and Step~2. Let $v^*_{U}$ denote the best upper bound of the scaled CVaR approximation \eqref{ccp_cvar_eq_scale}. For each $i\in [N]$, we solve $\eta_{i}=\min_{\bm x\in[0,1]^n}\{\bm c^{\top} \bm x: \sum_{k\in[n]}\xi^i_k x_k \geq 1,\sum_{k\in[n]}x_k\leq 0.2n\}$ . If $\eta_{i} > v^*_{U}$ for any $i\in [N]$, then we enforce $\alpha_i=1$. All other settings are consistent with those described in Section~\ref{sec_numerical_study_joint}. 
Table~\ref{tab_porfolio_summary} summarizes the average numerical performance across instances, whereas the detailed results under different choices of $\varepsilon$ and $N$ are provided in Table~\ref{tab_porfolio} in Appendix~\ref{appendix_numerical_study}. 
Overall, our approach for solving the scaled CVaR approximation \eqref{ccp_cvar_eq_scale} consistently yields higher-quality solutions than the $\alsox\#$ method. When averaged across all tested instances, the improvement achieved by the scaled CVaR approach increases with the violation level~$\varepsilon$, ranging from $2.43\%$ versus $1.70\%$ at $\varepsilon=0.050333$ to $9.02\%$ versus $7.80\%$ at $\varepsilon=0.300333$. These results indicate that the proposed method attains improved objective values, particularly for moderate-to-high risk parameters, while requiring only a modest increase in computational time.
}

\begin{table}[tb!]
\vspace{-1em}
\centering
\scriptsize
\caption{\rev{Numerical Results of Portfolio Optimization}}
\renewcommand{\arraystretch}{1}
\setlength{\tabcolsep}{2pt}
\label{tab_porfolio_summary}
\rev{
\begin{tabular}{cccrrrrrrr}
\hline
\multirow{2}{*}{$\varepsilon$} & \multirow{2}{*}{$n$} & \multirow{2}{*}{$N$} & \multirow{2}{*}{\makecell{CVaR \\ Approximation}} & \multicolumn{3}{c}{$\alsox\#$} & \multicolumn{3}{c}{Scaled CVaR Approximation} \\
\cline{5-10}
&  &  &  & Value & Time (s) & Improvement & Value & Time (s) & Improvement \\
\hline
{0.050333}
& {50}
& 500  & 2.86 & 2.80 & 0.38 & 1.76\% & 2.79 & 9.62  & 2.37\% \\
&      & 1000 & 2.84 & 2.79 & 0.70 & 1.63\% & 2.77 & 19.03 & 2.48\% \\
\hline
\hline
\multirow{2}{*}{0.100333}
& \multirow{2}{*}{50}
& 500  & 2.78 & 2.68 & 0.48 & 3.24\% & 2.68 & 9.91  & 3.66\% \\
&      & 1000 & 2.77 & 2.69 & 0.94 & 2.36\% & 2.67 & 20.55 & 3.54\% \\
\hline
\hline
\multirow{2}{*}{0.200333}
& \multirow{2}{*}{50}
& 500  & 2.68 & 2.53 & 0.64 & 5.08\% & 2.52 & 10.20 & 5.89\% \\
&      & 1000 & 2.67 & 2.52 & 1.23 & 5.00\% & 2.50 & 20.94 & 6.00\% \\
\hline
\hline
\multirow{2}{*}{0.300333}
& \multirow{2}{*}{50}
& 500  & 2.60 & 2.37 & 0.68 & 7.76\% & 2.36 & 10.41 & 8.92\% \\
&      & 1000 & 2.58 & 2.35 & 1.34 & 7.84\% & 2.33 & 21.79 & 9.12\% \\
\hline
\multicolumn{6}{r}{Average} & $\mathbf{4.33\%}$ &  &  & $\mathbf{5.25\%}$ \\
\hline
\multicolumn{7}{l}{\tiny *Each row is the average over five independent instances.}
\end{tabular}
}
\end{table}


\section{Conclusion}
\label{sec_conclusion}
In this paper, we investigated the scaled CVaR approximation for solving CCPs. We provided sufficient conditions under which the scaled CVaR approximation preserves an optimal solution of a CCP. We also developed efficient algorithms to solve the scaled CVaR approximation. Our numerical study confirmed the effectiveness of the proposed algorithms. \rev{An important direction for future research is the extension of the proposed scaling framework to broader distributionally robust settings, including general Wasserstein ambiguity sets, which may require new definitions of scaling.}

\bibliographystyle{plain}
\bibliography{scaled_cvar.bib}

\newpage
\appendix

\section{Examples}
\label{appendix_sec_example}

\begin{example}\rm
\label{example_motivate_covering}
Consider a single CCP with $2$ equiprobable scenarios (i.e., $p_1=p_2=1/2$), risk parameter $\varepsilon=2/3$, set $\X=\Re_+^2$, constraint $g(\bm x,\rxi)=1-{\rxi}^\top \bm{x}$, and $\rxi^1 =(1,0)^\top$, and $\rxi^2 =(1,1)^\top$. In this case, the CCP \eqref{eq_ccp_finite} is equivalent to the following mixed-integer linear program
\begin{align}\label{exmp}
v^*=\min_{\bm{x}\in\Re^2_+,\bm z \in\{0,1\}^2}\left\{ 2x_1+x_2\colon x_1\geq z_1, x_1+x_2\geq z_2, z_1+z_2\geq 1	\right\}.
\end{align}
Note that here we have $v^*=1$. The CVaR approximation of the CCP \eqref{exmp} is
\begin{align*}
v^\CVaR=\min_{x_1\geq0, x_2\geq0, \beta\leq 0,\bm s\geq \bm 0} \left\{2x_1+x_2: 
\begin{array}{l}
\displaystyle
x_1\geq 1- s_1-\beta, x_1+x_2\geq 1- s_2-\beta, \\
\displaystyle (s_1+s_2)/2+2\beta/3\leq 0
\end{array}
\right\},
\end{align*}
with $v^\CVaR=2$. 
To improve the CVaR approximation, we can implement a scaling procedure. To illustrate, note that, given $\alpha\geq 1$, $g(\bm x,\rxi)\leq 0$ is equivalent to $\bar{g}(\bm x,\rxi)\leq 0$ with $\bar{g}(\bm x,\rxi^1)=g(\bm x,\rxi^1)$ and $\bar{g}(\bm x,\rxi^2)=\rev{\alpha{g}(\bm x,\rxi^2)}$. In this case, one can replace $g(\bm x,\rxi)\leq 0$ by $\bar{g}(\bm x,\rxi)\leq 0$ in the original CCP \eqref{exmp}, resulting in the following CVaR approximation: 
\begin{align*}
v^\CVaR(\alpha)=\min_{x_1\geq0, x_2\geq0, \beta\leq 0,\bm s\geq \bm 0} \left\{2x_1+x_2: 
\begin{array}{l}
\displaystyle
x_1\geq 1- s_1-\beta, \alpha(x_1+x_2)\geq \alpha - s_2-\beta, \\
\displaystyle (s_1+s_2)/2+2\beta/3\leq 0
\end{array}
\right\}.
\end{align*}
Note that the optimal objective value $v^*$ is independent of the scaling factor $\alpha$ while $v^\CVaR(\alpha)$ varies with changes in $\alpha$.
When $\alpha\in[1,3)$, $v^\CVaR(\alpha)=2$ with an optimal solution $x_1^*=1, x_2^*=0, s_1^*=s_2^*=0, \beta^*=0$. When $\alpha \in(3,\infty)$, $v^\CVaR(\alpha)=1+3/\alpha$ with an optimal solution $x_1^*=0,x_2^*=1+\frac{3}{\alpha}, s_1^*=4,s_2^*=0,\beta^*=-3$. The optimal objective values $v^\CVaR(\alpha)$ of the scaled CVaR approximation are plotted in Figure~\ref{Fig_example_two_constraint}. As we may observe in this example, $v^*=\lim_{\alpha\rightarrow \infty}v^\CVaR(\alpha)$.
\begin{figure}[htbp]
\begin{center}
\centering
\includegraphics[width=0.5\textwidth]{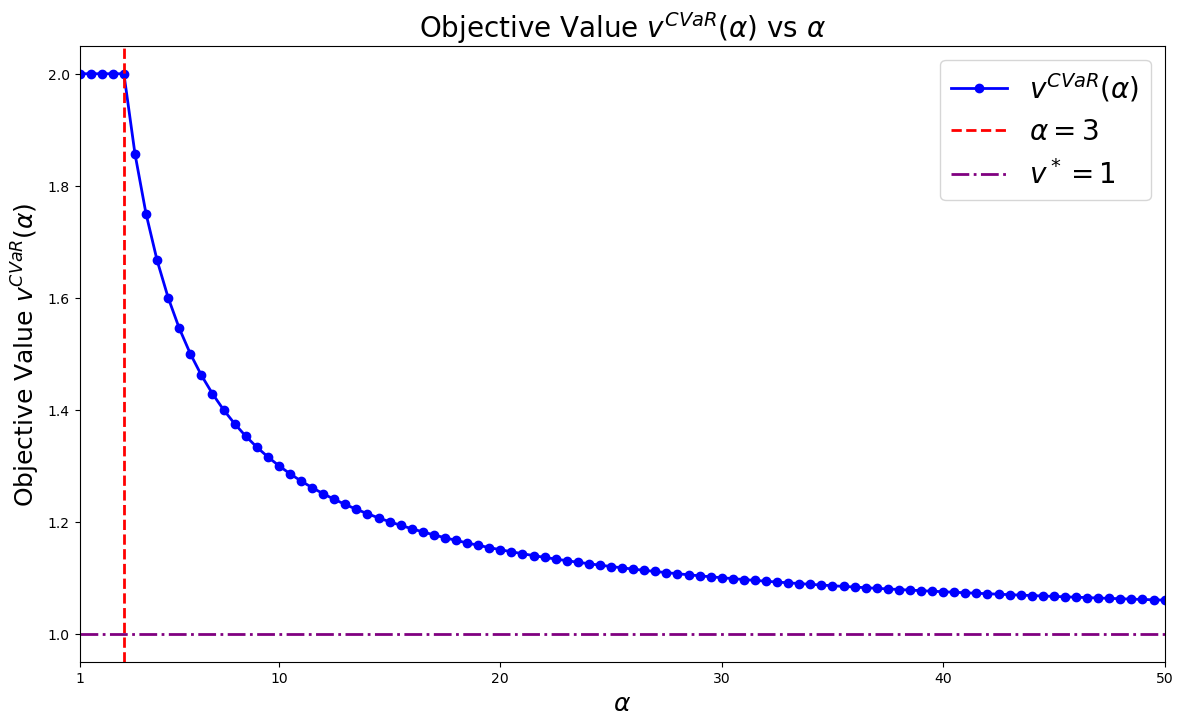}
\caption{The Scaled CVaR Approximation Objective $v^\CVaR(\alpha)$ for Scaling Factor $\alpha\in[1,50]$.}
\label{Fig_example_two_constraint}
\end{center}
\end{figure}\QEDB
\end{example}

\begin{example}\rm
\label{example_discrete_exact}
Consider a single CCP with $4$ equiprobable scenarios (i.e., $N=4$, $p_1=p_2=p_3=p_4=1/4$), risk parameter $\varepsilon=1/2$, deterministic set $\X=\Ze_+$, constraint $g(\bm x,\rxi)={\xi}_1 x - {\xi}_2$,
$\xi_1^1=-9$, $\xi_1^2=\xi_1^3=\xi_1^4=4$, $\xi_2^1=-10$, and $\xi_2^2=\xi_2^3=\xi_2^4=2$. In this example, CCP \eqref{eq_ccp_finite} resorts to
\begin{equation*}
v^*=\min_{x\in \Ze_+}\left\{-x\colon  \I\left(9x\geq 10\right)+\I\left(4x\leq 2\right)+\I\left( 4x\leq 2\right)+\I\left( 4x\leq 2\right)\geq 2 \right\}.
\end{equation*}
By a simple calculation, the optimal value is $v^*=0$ with the optimal solution $x^*=0$. Notice that with the solution $x^*=0$, we have (i) $4x^*=0< 2$; and (ii) $1/4[\I\left(9x^*\geq 10\right)+\I\left(4x^*\leq 2\right)+\I\left( 4x^*\leq 2\right)+\I\left( 4x^*\leq 2\right)]>1-1/2$. Hence, the two conditions in \Cref{theorem_general_case} are satisfied. The scaled CVaR approximation \eqref{ccp_cvar_eq_scale} is expected to provide the optimal solution in this example. We first check the corresponding CVaR approximation, that is,
\begin{align*}
v^\CVaR=\min_{x\in \Ze_+, \beta\leq 0,\bm s\geq \bm 0} \left\{-x: 
\begin{array}{l}
\displaystyle
-9x+10\leq  s_1, 4x-2\leq  s_2+\beta, 4x-2\geq  s_3+\beta, \\
\displaystyle  4x-2\leq  s_4+\beta, \sum_{i\in[4]}s_i/4+\beta/2\leq 0
\end{array}
\right\}.
\end{align*}
It turns out this CVaR approximation is infeasible.

We then consider the scaled CVaR approximation \eqref{ccp_cvar_eq_scale}, which can be recast as,
\begin{equation*}
\begin{aligned}
v_S^\CVaR=\inf_{\begin{subarray}{c}x\in \Ze_+, \beta\leq 0,\\\bm s\geq \bm 0,\bm \alpha\geq \bm e\end{subarray}} \left\{-x: 
\begin{array}{l}
\displaystyle
-9x \alpha_1+10\alpha_1\leq  s_1+\beta, 4x\alpha_2-2\alpha_2\leq  s_2+\beta, 4x\alpha_3-2\alpha_3\leq  s_3+\beta, \\
\displaystyle 4x\alpha_4-2\alpha_4\leq  s_4+\beta, \sum_{i\in[4]}s_i/4+\beta/2\leq 0
\end{array}
\right\}.
\end{aligned}
\end{equation*}
We obtain $v_S^\CVaR=0$ with $ x^*=0, \beta^*=-10, \alpha^*_1=1,\alpha^*_2=\alpha^*_3=\alpha^*_4=5, s^*_1=20, s^*_2=s^*_3=s^*_4=0$. Therefore, in this example, we can use the scaled CVaR approximation \eqref{ccp_cvar_eq_scale} to find the optimal solution. 
\QEDB
\end{example}

\begin{example}\rm
\label{example_discrete_not_exact_condition_1}
Consider a single CCP with $4$ equiprobable scenarios (i.e., $N=4$, $p_1=p_2=p_3=p_4=1/4$), risk parameter $\varepsilon=1/2$, deterministic set $\X=\Ze_+$, constraint $g(\bm x,\rxi)={\xi}_1 x - {\xi}_2$, $\xi_1^1=-9$, $\xi_1^2=\xi_1^3=\xi_1^4=4$, $\xi_2^1=-10$, and $\xi_2^2=\xi_2^3=2,\xi_2^4=0$. In this example, CCP \eqref{eq_ccp_finite} resorts to
\begin{equation*}
v^*=\min_{x\in \Ze_+}\left\{-x\colon  \I\left(9x\geq 10\right)+\I\left(4x\leq 2\right)+\I\left( 4x\leq 2\right)+\I\left( 4x\leq 0\right)\geq 2 \right\}.
\end{equation*}
By a simple calculation, the optimal value is $v^*=0$ with the optimal solution $x^*=0$. 
We also have $ 4x^*=0\not< 0$, which violates the first assumption of \Cref{theorem_general_case}. In this case, the scaled CVaR approximation \eqref{ccp_cvar_eq_scale} can be written as
\begin{align*}
v_S^\CVaR=\inf_{x\in \Ze_+, \beta\leq 0,\bm s\geq \bm 0,\bm \alpha\geq \bm e} \left\{-x: 
\begin{array}{l}
\displaystyle
-9x \alpha_1+10\alpha_1\leq  s_1+\beta, 4x\alpha_2-2\alpha_2\leq  s_2+\beta,\\
\displaystyle  4x\alpha_3-2\alpha_3\leq  s_3+\beta,  4x\alpha_4\leq  s_4+\beta, \sum_{i\in[4]}s_i/4+\beta/2\leq 0
\end{array}
\right\}.
\end{align*}
However, in this case, the scaled CVaR approximation is infeasible. This implies that when \rev{the condition $g(\bm x^*,\rxi^i)<0$ for each $i\in I^*$} in \Cref{theorem_general_case} is violated, the scaled 
CVaR approximation may not yield the optimal solution.
\QEDB
\end{example}

\begin{example}\rm
\label{example_discrete_not_exact_condition_2}
Consider a single CCP with $4$ equiprobable scenarios (i.e., $N=4$, $p_1=p_2=p_3=p_4=1/4$), risk parameter $\varepsilon=1/4$, deterministic set $\X=\Ze_+$, constraint $g(\bm x,\rxi)={\xi}_1 x - {\xi}_2$, $\xi_1^1=-9$, $\xi_1^2=\xi_1^3=\xi_1^4=4$, $\xi_2^1=-10$, and $\xi_2^2=\xi_2^3=\xi_2^4=2$. In this example, CCP \eqref{eq_ccp_finite} resorts to
\begin{equation*}
v^*=\min_{x\in \Ze_+}\left\{-x\colon  \I\left(9x\geq 10\right)+\I\left(4x\leq 2\right)+\I\left( 4x\leq 2\right)+\I\left( 4x\leq 2\right)\geq 3 \right\}.
\end{equation*}
By a simple calculation, the optimal value is $v^*=0$ with the optimal solution $x^*=0$. We also have $1/4[\I\left(9x^*\geq 10\right)+\I\left(4x^*\leq 2\right)+\I\left( 4x^*\leq 2\right)+\I\left( 4x^*\leq 2\right)]=3/4\not>1-1/4$, which violates the second assumption of \Cref{theorem_general_case}. In this case, the scaled CVaR approximation \eqref{ccp_cvar_eq_scale} can be written as
\begin{align*}
v_S^\CVaR=\inf_{x\in \Ze_+, \beta\leq 0,\bm s\geq \bm 0,\bm \alpha\geq \bm e} \left\{-x: 
\begin{array}{l}
\displaystyle
-9x \alpha_1+10\alpha_1\leq  s_1+\beta, 4x\alpha_2-2\alpha_2\leq  s_2+\beta, \\
\displaystyle  4x\alpha_3-2\alpha_3\leq  s_3+\beta, 4x\alpha_4-2\alpha_4\leq  s_4+\beta, \sum_{i\in[4]}s_i/4+\beta/4\leq 0
\end{array}
\right\}.
\end{align*}
However, this scaled CVaR approximation is infeasible, which implies that we cannot use the scaled CVaR approximation to find the optimal solution when we violate the \rev{condition $\sum_{i\in I^*}p_i >1-\varepsilon$} in \Cref{theorem_general_case}. 
\QEDB
\end{example}

\begin{example}\rm
\label{example_continuous_exact}
Consider a single CCP with $3$ equiprobable scenarios (i.e., $N=3$, $p_1=p_2=p_3=1/3$), risk parameter $\varepsilon=0.4$, set $\X=\Re_+^2$, constraint $g(\bm x,\rxi)=1-{\rxi}^\top \bm{x}$, and $\rxi^1 =(1,0)^\top$, $\rxi^2 =(1,1)^\top$, and $\rxi^3 =(1,1)^\top$. In this case, the CCP is equivalent to the following mixed-integer linear program
\begin{align*}
v^*=\min_{\bm{x}\in\Re^2_+,\bm z \in\{0,1\}^3}\left\{ 3x_1+2x_2\colon x_1\geq z_1, x_1+x_2\geq z_2,x_1+x_2\geq z_3, z_1+z_2+z_3\geq 2	\right\} =2.
\end{align*}
CVaR approximation is
\begin{align*}
v^\CVaR=\min_{\begin{subarray}{c} x_1\geq0, x_2\geq0,\\ \beta\leq 0,\bm s\geq \bm 0\end{subarray}} \left\{3x_1+2x_2: 
\begin{array}{l}
\displaystyle
x_1\geq 1- s_1-\beta, x_1+x_2\geq 1- s_2-\beta, \\
\displaystyle x_1+x_2\geq 1- s_3-\beta, \sum_{i\in[3]}s_i/3+0.4\beta\leq 0
\end{array}
\right\}.
\end{align*}
We have $v^\CVaR=3$. We then consider the scaled CVaR approximation \eqref{ccp_cvar_eq_scale}, that is,
\begin{align*}
v_S^\CVaR=\inf_{\begin{subarray}{c} x_1\geq0, x_2\geq0,\\
\beta\leq 0,\bm s\geq \bm 0,\bm \alpha\geq \bm e\end{subarray}} \left\{3x_1+2x_2: 
\begin{array}{l}
\displaystyle
x_1\alpha_1\geq \alpha_1-s_1-\beta, x_1\alpha_2+x_2\alpha_2\geq \alpha_2- s_2-\beta,\\
\displaystyle x_1\alpha_3+x_2\alpha_3\geq \alpha_3- s_3-\beta, \sum_{i\in[3]}s_i/3+0.4\beta\leq 0
\end{array}
\right\}.
\end{align*}
We obtain $v_S^\CVaR=2$. For $\alpha>10$, an optimal solution of the scaled CVaR approximation is $ x_1^*=0, x_2^*=1+5/\alpha^*_2,\beta^*=-5, \alpha_1=1,\alpha^*_2=\alpha^*_3, s^*_1=6, s^*_2=s^*_3=0$. Hence, $\lim_{\alpha^*_2\to\infty}x_2^*=1$.
Therefore, in this example, we can use the scaled CVaR approximation \eqref{ccp_cvar_eq_scale} to find the optimal solution.
\QEDB
\end{example}

\begin{example}\rm
\label{example_continuous_not_exact}
Revisit \Cref{example_continuous_exact} with 
the risk parameter $\varepsilon=1/3$. The scaled CVaR approximation \eqref{ccp_cvar_eq_scale} can be written as
\begin{align*}
v_S^\CVaR=\inf_{\begin{subarray}{c} x_1\geq0, x_2\geq0,\\
\beta\leq 0,\bm s\geq \bm 0,\bm \alpha\geq \bm e\end{subarray}} \left\{3x_1+2x_2: 
\begin{array}{l}
\displaystyle
x_1\alpha_1\geq \alpha_1- s_1-\beta, x_1\alpha_2+x_2\alpha_2\geq \alpha_2- s_2-\beta,\\
\displaystyle  x_1\alpha_3+x_2\alpha_3\geq \alpha_3- s_3-\beta, \sum_{i\in[3]}s_i/3+\beta/3\leq 0
\end{array}
\right\}.
\end{align*}
We obtain $v_S^\CVaR=3$ with $ x_1^*=1, x_2^*=0,\beta^*=0, \alpha_1^*=\alpha_2^*=\alpha_3^*=1, s_1^*=s_2^*=s_3^*=0$, which is the same solution from CVaR approximation. Hence, in this example, we cannot use the scaled CVaR approximation \eqref{ccp_cvar_eq_scale} to find the optimal solution.
\QEDB
\end{example}

\section{Scaling in $\alsoxs$}
\label{sec_scaling_alsoxs}

The $\alsoxs$ method with a bilevel structure, recently proposed by \cite{jiang2025also}, can be used to solve CCP \eqref{eq_ccp_finite}.  In the lower-level $\alsoxs$, we solve the $\CVaR$ loss function with a given upper bound of the upper-level objective value. We then check whether its optimal solution $\bm{x}^*$ satisfies the chance constraint. 
The upper-level $\alsoxs$ searches for the best upper bound of the objective value.  Formally, the output of $\alsoxs$ admits the form:

\begin{subequations}
\begin{algorithm}[htbp]
\caption{$\alsoxs$ \citep{jiang2025also}}
\label{alg_alsox_sharp}
\begin{algorithmic}[1]
\State \textbf{Input:} Let $\delta_A$ denote the stopping tolerance parameter, and $t_L$ and $t_U$ be the known lower and upper bounds of the optimal value of CCP \eqref{eq_ccp_finite}, respectively 
\While {$t_U-t_L>\delta_A$}
\State Let $t=(t_L+t_U)/2$ and $(\bm{x}^{\as},\beta^{\as})$ be an optimal solution of the lower-level $\alsoxs$ \eqref{eq_also_sharp_ori}:
\begin{align}
(\bm x^{\as},\beta^{\as})\in\arg\min _{\bm {x}\in \X,\beta\leq 0}\, \left\{ \varepsilon\beta+ \sum_{i\in[N]}p_i\left[  g(\bm x,\rxi^i)-\beta \right]_+ \colon \bm{c}^\top \bm{x} \leq t \right\}\label{eq_also_sharp_ori}
\end{align}
\State Let $t_L=t$ if $\bm x^{\as}$ satisfies the checking condition of the upper-level $\alsoxs$ \eqref{alsoxs_drccp_formualtion_upper}:
\begin{align}
\sum_{i\in[N]}p_i\I\left[ g(\bm x^{\as},\rxi^i)\leq 0\right] \geq 1-\varepsilon; \label{alsoxs_drccp_formualtion_upper}
\end{align}
\,\quad otherwise, $t_U=t$
\EndWhile
\State \textbf{Output:} A feasible solution $\bm x^{\as}$ and its objective value $t_U$ to CCP \eqref{eq_ccp_finite} 
\end{algorithmic}
\end{algorithm}
\end{subequations}
To enhance the solution quality of \Cref{alg_alsox_sharp}, we can integrate $\alsoxs$ with the scaling procedure proposed in  \Cref{alg_cvar_scale}. In the scaled $\alsoxs$, we first execute $\alsoxs$ \Cref{alg_alsox_sharp}, and when Step 4 of $\alsoxs$ \Cref{alg_alsox_sharp}  encounters an infeasible solution (i.e., the optimal solution $\bm{x}^{\as}$ of the lower-level $\alsoxs$ \eqref{eq_also_sharp_ori} violates the chance constraint $\sum_{i\in[N]}p_i\I\left[ g(\bm x^{\as},\rxi^i)\leq 0\right] < 1-\varepsilon$), then we run the scaling procedure in \Cref{alg_cvar_scale} with the same $t$ to update the coefficient of the constraint $g(\bm x,\cdot)$ and see if we can find a feasible solution. If YES, we further decrease $t_U=t$; otherwise, we increase $t_L=t$. The detailed procedure for the scaled $\alsoxs$ algorithm is shown in \Cref{alg_alsox_sharp_scaled}.

\begin{subequations}
\begin{algorithm}[htbp]
\caption{Scaled $\alsoxs$}
\label{alg_alsox_sharp_scaled}
\begin{algorithmic}[1]
\State \textbf{Input:} Let $\delta_A$ denote the stopping tolerance parameter, and $t_L$ and $t_U$ be the known lower and upper bounds of the optimal value of CCP \eqref{eq_ccp_finite}, respectively 
\While {$t_U-t_L>\delta_A$}
\State Let $t=(t_L+t_U)/2$ and $(\bm{x}^{\as},\beta^{\as})$ be an optimal solution of the lower-level $\alsoxs$ \eqref{eq_also_sharp_ori_1}:
\begin{align}
(\bm x^{\as},\beta^{\as})\in\arg\min _{\bm {x}\in \X,\beta\leq 0}\, \left\{ \varepsilon\beta+ \sum_{i\in[N]}p_i\left[  g(\bm x,\rxi^i)-\beta \right]_+ \colon \bm{c}^\top \bm{x} \leq t \right\}\label{eq_also_sharp_ori_1}
\end{align}
\State Let $t_L=t$ if $\bm x^{\as}$ satisfies the checking condition of the upper-level $\alsoxs$ \eqref{alsoxs_drccp_formualtion_upper}, i.e., $\sum_{i\in[N]}p_i\I\left[ g(\bm x^{\as},\rxi^i)\leq 0\right] \geq1-\varepsilon$; otherwise, run the scaling procedure in \Cref{alg_cvar_scale} to update the coefficient of the constraint $g(\bm x,\cdot)$. If the solution output from the scaling procedure in \Cref{alg_cvar_scale} is feasible to the CCP, let $t_L=t$; otherwise, $t_U=t$
\EndWhile
\State \textbf{Output:} A feasible solution $\bar{\bm x}^{\as}$ and its objective value $t_U$ to CCP \eqref{eq_ccp_finite} 
\end{algorithmic}
\end{algorithm}
\end{subequations}

We make the following remarks about the scaled $\alsoxs$ algorithm \Cref{alg_alsox_sharp_scaled}:
\begin{enumerate} [label=(\roman*)]
\item Formally, the scaled $\alsoxs$ \Cref{alg_alsox_sharp_scaled} enhances the solution quality of $\alsoxs$ \Cref{alg_alsox_sharp};
\item We can use the solutions of the the lower-level $\alsoxs$ \eqref{eq_also_sharp_ori_1} as warm-starts for the scaling procedure in \Cref{alg_alsox_sharp_scaled}. However, the detailed implementation of \Cref{alg_alsox_sharp_scaled} is beyond the scope of this paper and is left for future research.
\end{enumerate}

\section{\rev{Detailed Numerical Results in Section~\ref{sec_numerical_study}}}
\label{appendix_numerical_study}
\rev{In this section, we present the detailed numerical results from Section~\ref{sec_numerical_study}, reported in Tables~\ref{tab_joint_eps_020} and~\ref{tab_porfolio}.}

\begin{table}[htbp]
\vspace{-1em}
\centering
\scriptsize
\caption{Numerical Results of a Joint CCP with \rev{Different} $\varepsilon$}
\renewcommand{\arraystretch}{1} 
\label{tab_joint_eps_020}
\setlength{\tabcolsep}{2pt}
\rev{
\begin{tabular}{ccccrrrrrrr}
\hline
\multirow{2}{*}{$\varepsilon$} & \multirow{2}{*}{$n$} & \multirow{2}{*}{$J$} & \multirow{2}{*}{Instance} & \multirow{2}{*}{\makecell{CVaR \\ Approximation}} & \multicolumn{3}{c}{$\alsox\#$} & \multicolumn{3}{c}{Scaled CVaR Approximation} \\
\cline{6-11}
&  &  &  &  & Value & Time (s) & Improvement & Value & Time (s) & Improvement \\
\hline
\multirow{15}{*}{$0.050333$}
& \multirow{5}{*}{$20$} & \multirow{5}{*}{$10$}
& 1-4-multi-3000-1 & $-5789.51$ & $-5843.79$ & $82.49$ & $0.94\%$ & $-5855.93$ & $514.46$ & $1.15\%$ \\
&  &  & 1-4-multi-3000-2 & $-5779.51$ & $-5846.49$ & $82.29$ & $1.16\%$ & $-5847.95$ & $516.39$ & $1.18\%$ \\
&  &  & 1-4-multi-3000-3 & $-5769.41$ & $-5843.38$ & $82.27$ & $1.28\%$ & $-5851.34$ & $514.92$ & $1.42\%$ \\
&  &  & 1-4-multi-3000-4 & $-5783.57$ & $-5851.97$ & $86.89$ & $1.18\%$ & $-5852.95$ & $544.46$ & $1.20\%$ \\
&  &  & 1-4-multi-3000-5 & $-5770.51$ & $-5842.69$ & $85.23$ & $1.25\%$ & $-5847.29$ & $535.01$ & $1.33\%$ \\
\cline{2-11}
& \multirow{5}{*}{$39$} & \multirow{5}{*}{$5$}
& 1-6-multi-3000-1 & $-9847.60$ & $-9959.17$ & $84.92$ & $1.13\%$ & $-9955.26$ & $519.66$ & $1.09\%$ \\
&  &  & 1-6-multi-3000-2 & $-9866.40$ & $-9974.94$ & $83.07$ & $1.10\%$ & $-9981.91$ & $527.56$ & $1.17\%$ \\
&  &  & 1-6-multi-3000-3 & $-9871.75$ & $-9977.33$ & $82.69$ & $1.07\%$ & $-9984.87$ & $510.72$ & $1.15\%$ \\
&  &  & 1-6-multi-3000-4 & $-9883.64$ & $-10003.54$ & $82.87$ & $1.21\%$ & $-10000.99$ & $537.30$ & $1.19\%$ \\
&  &  & 1-6-multi-3000-5 & $-9852.74$ & $-9968.08$ & $85.73$ & $1.17\%$ & $-9975.77$ & $556.52$ & $1.25\%$ \\
\cline{2-11}
& \multirow{5}{*}{$50$} & \multirow{5}{*}{$5$}
& 1-7-multi-3000-1 & $-15778.48$ & $-15889.29$ & $112.51$ & $0.70\%$ & $-15887.76$ & $702.28$ & $0.69\%$ \\
&  &  & 1-7-multi-3000-2 & $-15750.66$ & $-15865.71$ & $105.56$ & $0.73\%$ & $-15875.14$ & $676.66$ & $0.79\%$ \\
&  &  & 1-7-multi-3000-3 & $-15782.39$ & $-15890.83$ & $104.74$ & $0.69\%$ & $-15898.61$ & $665.38$ & $0.74\%$ \\
&  &  & 1-7-multi-3000-4 & $-15774.55$ & $-15892.31$ & $109.25$ & $0.75\%$ & $-15904.19$ & $709.18$ & $0.82\%$ \\
&  &  & 1-7-multi-3000-5 & $-15744.71$ & $-15860.44$ & $108.82$ & $0.74\%$ & $-15863.01$ & $697.77$ & $0.75\%$ \\
\hline
\multicolumn{7}{r}{Average} & $\mathbf{1.01}\%$ &  &  & $\mathbf{1.06}\%$ \\
\hline
\multirow{15}{*}{$0.100333$}
& \multirow{5}{*}{$20$} & \multirow{5}{*}{$10$}
& 1-4-multi-3000-1 & $-5831.68$ & $-5896.96$ & $83.84$ & $1.12\%$ & $-5905.09$ & $528.52$ & $1.26\%$ \\
&  &  & 1-4-multi-3000-2 & $-5828.17$ & $-5905.42$ & $84.64$ & $1.33\%$ & $-5911.56$ & $531.10$ & $1.43\%$ \\
&  &  & 1-4-multi-3000-3 & $-5821.63$ & $-5904.40$ & $82.41$ & $1.42\%$ & $-5905.65$ & $523.53$ & $1.44\%$ \\
&  &  & 1-4-multi-3000-4 & $-5829.83$ & $-5904.91$ & $85.21$ & $1.29\%$ & $-5910.76$ & $541.50$ & $1.39\%$ \\
&  &  & 1-4-multi-3000-5 & $-5819.61$ & $-5895.10$ & $85.31$ & $1.30\%$ & $-5894.11$ & $532.75$ & $1.28\%$ \\
\cline{2-11}
& \multirow{5}{*}{$39$} & \multirow{5}{*}{$5$}
& 1-6-multi-3000-1 & $-9925.97$ & $-10051.00$ & $82.90$ & $1.26\%$ & $-10058.75$ & $526.54$ & $1.34\%$ \\
&  &  & 1-6-multi-3000-2 & $-9945.20$ & $-10074.19$ & $83.89$ & $1.30\%$ & $-10083.20$ & $540.89$ & $1.39\%$ \\
&  &  & 1-6-multi-3000-3 & $-9956.57$ & $-10094.98$ & $84.01$ & $1.39\%$ & $-10093.75$ & $536.86$ & $1.38\%$ \\
&  &  & 1-6-multi-3000-4 & $-9970.29$ & $-10116.59$ & $85.54$ & $1.47\%$ & $-10119.14$ & $558.59$ & $1.49\%$ \\
&  &  & 1-6-multi-3000-5 & $-9937.59$ & $-10069.96$ & $84.41$ & $1.33\%$ & $-10078.05$ & $529.11$ & $1.41\%$ \\
\cline{2-11}
& \multirow{5}{*}{$50$} & \multirow{5}{*}{$5$}
& 1-7-multi-3000-1 & $-15863.65$ & $-16007.58$ & $104.56$ & $0.91\%$ & $-16011.72$ & $673.40$ & $0.93\%$ \\
&  &  & 1-7-multi-3000-2 & $-15838.68$ & $-15984.78$ & $108.66$ & $0.92\%$ & $-15999.82$ & $691.10$ & $1.02\%$ \\
&  &  & 1-7-multi-3000-3 & $-15864.52$ & $-15995.88$ & $104.93$ & $0.83\%$ & $-16006.65$ & $679.16$ & $0.90\%$ \\
&  &  & 1-7-multi-3000-4 & $-15860.93$ & $-16006.49$ & $112.01$ & $0.92\%$ & $-16019.04$ & $711.60$ & $1.00\%$ \\
&  &  & 1-7-multi-3000-5 & $-15833.63$ & $-15982.70$ & $109.83$ & $0.94\%$ & $-15987.32$ & $700.24$ & $0.97\%$ \\
\hline
\multicolumn{7}{r}{Average} & $\mathbf{1.18}\%$ &  &  & $\mathbf{1.24}\%$ \\
\hline
\multirow[0]{15}{*}{$0.200333$} & \multirow[0]{5}{*}{$20$} & \multirow[0]{5}{*}{$10$} & 1-4-multi-3000-1 & $-5882.23$ & $-5970.88$ & $86.58$ & $1.51\%$ & $-5980.71$ & $577.25$ & $1.67\%$ \\
&  &  & 1-4-multi-3000-2 & $-5882.54$ & $-5971.24$ & $86.85$ & $1.51\%$ & $-5973.87$ & $640.85$ & $1.55\%$ \\
&  &  & 1-4-multi-3000-3 & $-5878.65$ & $-5972.94$ & $143.45$ & $1.60\%$ & $-5978.93$ & $906.55$ & $1.71\%$ \\
&  &  & 1-4-multi-3000-4 & $-5885.45$ & $-5975.44$ & $86.07$ & $1.53\%$ & $-5986.93$ & $569.93$ & $1.72\%$ \\
&  &  & 1-4-multi-3000-5 & $-5873.65$ & $-5962.76$ & $95.51$ & $1.52\%$ & $-5967.68$ & $609.77$ & $1.60\%$ \\
\cline{2-11}
& \multirow[0]{5}{*}{$39$} & \multirow[0]{5}{*}{$5$} & 1-6-multi-3000-1 & $-10021.39$ & $-10186.63$ & $86.27$ & $1.65\%$ & $-10194.21$ & $545.36$ & $1.72\%$ \\
&  &  & 1-6-multi-3000-2 & $-10041.39$ & $-10202.49$ & $86.25$ & $1.60\%$ & $-10213.07$ & $545.86$ & $1.71\%$ \\
&  &  & 1-6-multi-3000-3 & $-10060.69$ & $-10239.33$ & $84.70$ & $1.78\%$ & $-10241.54$ & $540.22$ & $1.80\%$ \\
&  &  & 1-6-multi-3000-4 & $-10077.44$ & $-10258.85$ & $85.58$ & $1.80\%$ & $-10261.55$ & $551.13$ & $1.83\%$ \\
&  &  & 1-6-multi-3000-5 & $-10037.88$ & $-10209.94$ & $85.31$ & $1.71\%$ & $-10221.13$ & $536.61$ & $1.83\%$ \\
\cline{2-11}
& \multirow[0]{5}{*}{$50$} & \multirow[0]{5}{*}{$5$} & 1-7-multi-3000-1 & $-15971.22$ & $-16145.13$ & $107.16$ & $1.09\%$ & $-16158.38$ & $698.58$ & $1.17\%$ \\
&  &  & 1-7-multi-3000-2 & $-15953.23$ & $-16146.98$ & $108.54$ & $1.21\%$ & $-16153.00$ & $688.87$ & $1.25\%$ \\
&  &  & 1-7-multi-3000-3 & $-15966.21$ & $-16137.15$ & $105.53$ & $1.07\%$ & $-16152.76$ & $665.84$ & $1.17\%$ \\
&  &  & 1-7-multi-3000-4 & $-15973.34$ & $-16161.04$ & $108.43$ & $1.18\%$ & $-16152.64$ & $693.69$ & $1.12\%$ \\
&  &  & 1-7-multi-3000-5 & $-15948.75$ & $-16138.35$ & $107.20$ & $1.19\%$ & $-16149.07$ & $693.50$ & $1.26\%$ \\
\hline
\multicolumn{7}{r}{Average} & $\mathbf{1.46}\%$ &  &  & $\mathbf{1.54}\%$ \\
\hline
\multirow{15}{*}{$0.300333$} & \multirow{5}{*}{$20$} & \multirow{5}{*}{$10$} & 1-4-multi-3000-1 & $-5918.46$ & $-6021.46$ & $83.16$ & $1.74\%$ & $-6035.26$ & $551.97$ & $1.97\%$ \\
&  &  & 1-4-multi-3000-2 & $-5918.62$ & $-6021.50$ & $82.34$ & $1.74\%$ & $-6031.71$ & $565.63$ & $1.91\%$ \\
&  &  & 1-4-multi-3000-3 & $-5917.11$ & $-6026.06$ & $84.21$ & $1.84\%$ & $-6030.98$ & $558.55$ & $1.92\%$ \\
&  &  & 1-4-multi-3000-4 & $-5923.33$ & $-6026.11$ & $85.33$ & $1.74\%$ & $-6034.92$ & $576.17$ & $1.88\%$ \\
&  &  & 1-4-multi-3000-5 & $-5911.10$ & $-6017.82$ & $84.42$ & $1.81\%$ & $-6023.47$ & $566.99$ & $1.90\%$ \\
\cline{2-11}
& \multirow{5}{*}{$39$} & \multirow{5}{*}{$5$} & 1-6-multi-3000-1 & $-10091.29$ & $-10287.13$ & $84.76$ & $1.94\%$ & $-10291.76$ & $545.71$ & $1.99\%$ \\
&  &  & 1-6-multi-3000-2 & $-10110.44$ & $-10303.21$ & $87.72$ & $1.91\%$ & $-10312.85$ & $548.39$ & $2.00\%$ \\
&  &  & 1-6-multi-3000-3 & $-10136.15$ & $-10344.58$ & $84.32$ & $2.06\%$ & $-10342.07$ & $527.24$ & $2.03\%$ \\
&  &  & 1-6-multi-3000-4 & $-10154.42$ & $-10375.44$ & $85.69$ & $2.18\%$ & $-10384.77$ & $543.57$ & $2.27\%$ \\
&  &  & 1-6-multi-3000-5 & $-10112.49$ & $-10319.31$ & $84.94$ & $2.05\%$ & $-10333.10$ & $543.62$ & $2.18\%$ \\
\cline{2-11}
& \multirow{5}{*}{$50$} & \multirow{5}{*}{$5$} & 1-7-multi-3000-1 & $-16046.86$ & $-16258.91$ & $107.78$ & $1.32\%$ & $-16266.73$ & $689.61$ & $1.37\%$ \\
&  &  & 1-7-multi-3000-2 & $-16035.80$ & $-16267.09$ & $112.93$ & $1.44\%$ & $-16277.78$ & $715.99$ & $1.51\%$ \\
&  &  & 1-7-multi-3000-3 & $-16043.21$ & $-16256.97$ & $104.97$ & $1.33\%$ & $-16270.95$ & $666.46$ & $1.42\%$ \\
&  &  & 1-7-multi-3000-4 & $-16050.82$ & $-16265.87$ & $105.29$ & $1.34\%$ & $-16268.84$ & $681.97$ & $1.36\%$ \\
&  &  & 1-7-multi-3000-5 & $-16030.74$ & $-16255.95$ & $109.07$ & $1.40\%$ & $-16259.88$ & $678.48$ & $1.43\%$ \\
\hline
\multicolumn{7}{r}{Average} & $\mathbf{1.72}\%$ &  &  & $\mathbf{1.81}\%$ \\
\hline
\end{tabular}
}
\end{table}

\begin{table}[htbp]
\vspace{-1em}
\centering
\scriptsize
\caption{\rev{Numerical Results of Portfolio Optimization}}
\renewcommand{\arraystretch}{1} 
\setlength{\tabcolsep}{2pt}
\label{tab_porfolio}
\rev{
\begin{tabular}{ccccrrrrrrrr}
\hline
\multirow{2}{*}{$\varepsilon$} & \multirow{2}{*}{$n$} & \multirow{2}{*}{$N$} &  \multirow{2}{*}{\makecell{Instance \\ with Seeds}} & \multirow{2}{*}{\makecell{CVaR \\ Approximation}} & \multicolumn{3}{c}{$\alsox\#$} & \multicolumn{3}{c}{Scaled CVaR Approximation} \\
\cline{6-11}
&  & &  &  & Value & Time (s) & Improvement & Value & Time (s) & Improvement \\
\hline
\multirow{10}{*}{0.050333}                   & \multirow{10}{*}{50}                   & \multirow{5}{*}{500}                   & 1231                                                     & 4.18                                      & 4.08                      & 0.53                         & 2.33\%                          & 4.10                      & 9.89                         & 1.91\%                          \\
&                                        &                                        & 1232                                                     & 1.19                                      & 1.19                      & 0.28                         & 0.00\%                          & 1.17                      & 9.30                         & 1.78\%                          \\
&                                        &                                        & 1233                                                     & 2.61                                      & 2.58                      & 0.29                         & 1.26\%                          & 2.56                      & 9.54                         & 1.88\%                          \\
&                                        &                                        & 1234                                                     & 1.63                                      & 1.60                      & 0.28                         & 2.22\%                          & 1.58                      & 9.45                         & 3.36\%                          \\
&                                        &                                        & 1235                                                     & 4.69                                      & 4.55                      & 0.52                         & 3.01\%                          & 4.55                      & 9.94                         & 2.95\%                          \\
\cline{3-11}
&                                        & \multirow{5}{*}{1000}                  & 1231                                                     & 4.18                                      & 4.09                      & 1.00                         & 2.16\%                          & 4.10                      & 19.60                        & 1.90\%                          \\
&                                        &                                        & 1232                                                     & 1.19                                      & 1.19                      & 0.50                         & 0.00\%                          & 1.16                      & 18.47                        & 2.01\%                          \\
&                                        &                                        & 1233                                                     & 2.60                                      & 2.57                      & 0.51                         & 1.12\%                          & 2.55                      & 19.16                        & 1.89\%                          \\
&                                        &                                        & 1234                                                     & 1.62                                      & 1.59                      & 0.50                         & 1.84\%                          & 1.57                      & 18.34                        & 3.09\%                          \\
&                                        &                                        & 1235                                                     & 4.63                                      & 4.49                      & 1.00                         & 3.02\%                          & 4.47                      & 19.57                        & 3.53\%                          \\
\hline
\multicolumn{7}{r}{Average}                                               & $\mathbf{1.70\%}$                          &                           &                              & $\mathbf{2.43\%}$                          \\
\hline
\multirow{10}{*}{0.100333}                   & \multirow{10}{*}{50}                   & \multirow{5}{*}{500}                   & 1231                                                     & 4.08                                      & 3.92                      & 0.77                         & 3.97\%                          & 3.93                      & 10.34                        & 3.65\%                          \\
&                                        &                                        & 1232                                                     & 1.17                                      & 1.13                      & 0.27                         & 3.47\%                          & 1.13                      & 9.41                         & 3.00\%                          \\
&                                        &                                        & 1233                                                     & 2.56                                      & 2.53                      & 0.30                         & 1.37\%                          & 2.50                      & 9.96                         & 2.27\%                          \\
&                                        &                                        & 1234                                                     & 1.57                                      & 1.53                      & 0.28                         & 2.60\%                          & 1.50                      & 9.41                         & 4.51\%                          \\
&                                        &                                        & 1235                                                     & 4.54                                      & 4.32                      & 0.79                         & 4.79\%                          & 4.32                      & 10.44                        & 4.88\%                          \\
\cline{3-11}
&                                        & \multirow{5}{*}{1000}                  & 1231                                                     & 4.09                                      & 3.97                      & 1.51                         & 3.10\%                          & 3.97                      & 20.92                        & 2.96\%                          \\
&                                        &                                        & 1232                                                     & 1.16                                      & 1.16                      & 0.55                         & 0.00\%                          & 1.13                      & 20.28                        & 2.90\%                          \\
&                                        &                                        & 1233                                                     & 2.55                                      & 2.52                      & 0.51                         & 1.39\%                          & 2.49                      & 18.91                        & 2.34\%                          \\
&                                        &                                        & 1234                                                     & 1.56                                      & 1.51                      & 0.53                         & 2.67\%                          & 1.49                      & 20.75                        & 4.53\%                          \\
&                                        &                                        & 1235                                                     & 4.47                                      & 4.26                      & 1.60                         & 4.65\%                          & 4.25                      & 21.89                        & 4.94\%                          \\
\hline
\multicolumn{7}{r}{Average}                                        & $\mathbf{2.80\%}$                          &                           &                              & $\mathbf{3.60\%}$                          \\
\hline
\multirow{10}{*}{0.200333}                   & \multirow{10}{*}{50}                   & \multirow{5}{*}{500}                   & 1231                                                     & 3.94                                      & 3.70                      & 0.82                         & 6.15\%                          & 3.71                      & 10.58                        & 5.77\%                          \\
&                                        &                                        & 1232                                                     & 1.14                                      & 1.11                      & 0.53                         & 2.42\%                          & 1.09                      & 10.04                        & 4.38\%                          \\
&                                        &                                        & 1233                                                     & 2.50                                      & 2.41                      & 0.53                         & 3.69\%                          & 2.39                      & 9.96                         & 4.56\%                          \\
&                                        &                                        & 1234                                                     & 1.50                                      & 1.41                      & 0.52                         & 6.15\%                          & 1.39                      & 9.89                         & 7.39\%                          \\
&                                        &                                        & 1235                                                     & 4.34                                      & 4.04                      & 0.78                         & 6.96\%                          & 4.02                      & 10.52                        & 7.33\%                          \\
\cline{3-11}
&                                        & \multirow{5}{*}{1000}                  & 1231                                                     & 3.97                                      & 3.76                      & 1.57                         & 5.27\%                          & 3.74                      & 21.74                        & 5.84\%                          \\
&                                        &                                        & 1232                                                     & 1.13                                      & 1.11                      & 1.01                         & 2.35\%                          & 1.09                      & 20.58                        & 4.30\%                          \\
&                                        &                                        & 1233                                                     & 2.49                                      & 2.40                      & 1.01                         & 3.71\%                          & 2.38                      & 20.16                        & 4.45\%                          \\
&                                        &                                        & 1234                                                     & 1.48                                      & 1.38                      & 1.06                         & 6.82\%                          & 1.36                      & 20.98                        & 8.27\%                          \\
&                                        &                                        & 1235                                                     & 4.27                                      & 3.97                      & 1.52                         & 6.85\%                          & 3.96                      & 21.26                        & 7.12\%                          \\
\hline
\multicolumn{7}{r}{Average}                     & $\mathbf{5.04\%}$                          &                           &                              & $\mathbf{5.94\%}$                          \\
\hline
\multirow{10}{*}{0.300333}                   & \multirow{10}{*}{50}                   & \multirow{5}{*}{500}                   & 1231                                                     & 3.82                                      & 3.48                      & 0.78                         & 9.13\%                          & 3.47                      & 10.90                        & 9.15\%                          \\
&                                        &                                        & 1232                                                     & 1.11                                      & 1.08                      & 0.53                         & 2.95\%                          & 1.05                      & 10.06                        & 5.51\%                          \\
&                                        &                                        & 1233                                                     & 2.44                                      & 2.30                      & 0.53                         & 5.69\%                          & 2.27                      & 10.01                        & 7.00\%                          \\
&                                        &                                        & 1234                                                     & 1.43                                      & 1.29                      & 0.52                         & 10.16\%                         & 1.25                      & 10.01                        & 12.51\%                         \\
&                                        &                                        & 1235                                                     & 4.18                                      & 3.73                      & 1.04                         & 10.87\%                         & 3.74                      & 11.08                        & 10.45\%                         \\
\cline{3-11}
&                                        & \multirow{5}{*}{1000}                  & 1231                                                     & 3.85                                      & 3.51                      & 1.59                         & 8.70\%                          & 3.50                      & 22.32                        & 8.98\%                          \\
&                                        &                                        & 1232                                                     & 1.11                                      & 1.08                      & 1.00                         & 2.92\%                          & 1.05                      & 20.76                        & 5.31\%                          \\
&                                        &                                        & 1233                                                     & 2.43                                      & 2.29                      & 1.01                         & 5.62\%                          & 2.26                      & 20.80                        & 7.09\%                          \\
&                                        &                                        & 1234                                                     & 1.41                                      & 1.27                      & 1.05                         & 10.38\%                         & 1.23                      & 21.22                        & 12.82\%                         \\
&                                        &                                        & 1235                                                     & 4.10                                      & 3.62                      & 2.06                         & 11.56\%                         & 3.63                      & 23.89                        & 11.38\%                         \\
\hline
\multicolumn{7}{r}{Average}                                                & $\mathbf{7.80\%}$                          &                           &                              & $\mathbf{9.02\%}$   
\\
\hline
\end{tabular}
}
\end{table}

\section{Numerical Results Using IPOPT Solver (\citealt{wachter2006implementation})}
\label{appendix_sec_numerical_ipopt}

To further evaluate the performance of our method discussed in Section~\ref{sec_numerical_study}, we compare it with the IPOPT solver (see, e.g., \citealt{wachter2006implementation}) for directly solving the scaled CVaR approximation \eqref{ccp_cvar_eq_scale}. We consider solving the scaled CVaR approximation \eqref{ccp_cvar_eq_scale} under different upper-bound settings, that is, we consider $\alpha_U \in\{ 10000, 20000, 50000\}$ in the scaled CVaR approximation \eqref{ccp_cvar_eq_scale} as
\begin{align*}
v^{\CVaR}_S=\min_{\bm x\in \X,\beta\leq 0,\bm s\geq \bm 0,\bm\alpha} \left\{ \bm{c}^\top\bm{x}\colon \varepsilon\beta+ \sum_{i\in[N]}p_i s_i\leq 0, s_i+\beta\geq \alpha_i g(\bm x,\rxi^i),~i\in[N], 1\leq \alpha_i\leq \alpha_U,~i\in[N]\right\}.
\end{align*}
For all numerical results presented in this section, we use the solution of CVaR approximation \eqref{ccp_cvar_eq_1} as the initial starting point for the IPOPT solver. Each instance is solved with a time limit of $3600$ seconds, using the default settings described in \cite{wachter2006implementation}. If we cannot solve the instance within the time limit, we denote the result as ``---." We use ``\textrm{Improvement}'' to denote the percentage of differences between the value obtained from the IPOPT solver and CVaR approximation, i.e.,
\begin{align*}
\textrm{Improvement }(\%) = \frac{ \CVaR \textrm{ approximation value}-\textrm{value of the IPOPT solver} }{|\CVaR \textrm{ approximation value}|}\times 100\%.
\end{align*}
Since the solution of the IPOPT solver may converge to a local optimum that is worse than the initial solution, the Improvement reported may be negative. If IPOPT fails to find feasible solutions within the time limit, the average (marked with parentheses) will be calculated over instances for which IPOPT does find a feasible solution.

Following the same settings as those described in Section~\ref{sec_numerical_study_joint}, the results obtained using the IPOPT solver are presented in Table~\ref{tab_joint_ipopt_eps_020}. 
We find that the IPOPT solver is unstable when solving the scaled CVaR approximation \eqref{ccp_cvar_eq_scale} and the performance of IPOPT is sensitive to the choice of the upper bound of $\bm\alpha$. Regardless of the upper bound on $\bm\alpha$, the IPOPT solver fails to solve all instances to local optima. 

\begin{table}[htbp]
\centering
\tiny
\caption{Numerical Results of a Joint CCP with \rev{Different} $\varepsilon$ Using the IPOPT Solver}
\renewcommand{\arraystretch}{1}
\label{tab_joint_ipopt_eps_020}
\setlength{\tabcolsep}{1pt}
\rev{
\begin{tabular}{ccccrrrrrrrrrr}
\hline
\multirow{2}{*}{$\varepsilon$} & \multirow{2}{*}{$n$} & \multirow{2}{*}{$J$} & \multirow{2}{*}{Instance} & \multirow{2}{*}{\makecell{CVaR \\ Approximation}} & \multicolumn{3}{c}{IPOPT with $\alpha_U = 10000$}        & \multicolumn{3}{c}{IPOPT with $\alpha_U = 20000$}        & \multicolumn{3}{c}{IPOPT with $\alpha_U = 50000$}        \\ \cline{6-14} 
&                       &                       &                           &                         & \multicolumn{1}{c}{Value} & \multicolumn{1}{c}{Improvement} & \multicolumn{1}{c}{Time (s)} & \multicolumn{1}{c}{Value} & \multicolumn{1}{c}{Improvement} & \multicolumn{1}{c}{Time (s)} & \multicolumn{1}{c}{Value} & \multicolumn{1}{c}{Improvement} & \multicolumn{1}{c}{Time (s)} \\ \hline
\multirow{15}{*}{$0.050333$} & \multirow{5}{*}{$20$} & \multirow{5}{*}{$10$} & 1-4-multi-3000-1 & $-5789.51$ & \multicolumn{1}{r}{$-5856.36$} & \multicolumn{1}{r}{$1.15\%$} & $42.98$ & \multicolumn{1}{r}{$-3907.21$} & \multicolumn{1}{r}{$-32.51\%$} & $896.29$ & \multicolumn{1}{r}{$-3930.35$} & \multicolumn{1}{r}{$-32.11\%$} & $1137.78$ \\
& & & 1-4-multi-3000-2 & $-5779.51$ & \multicolumn{1}{r}{$-3774.00$} & \multicolumn{1}{r}{$-34.70\%$} & $1081.53$ & \multicolumn{1}{r}{$-3914.96$} & \multicolumn{1}{r}{$-32.26\%$} & $1172.65$ & \multicolumn{1}{r}{$-5847.85$} & \multicolumn{1}{r}{$1.18\%$} & $46.24$ \\
& & & 1-4-multi-3000-3 & $-5769.41$ & \multicolumn{1}{r}{$-5850.97$} & \multicolumn{1}{r}{$1.41\%$} & $36.34$ & \multicolumn{1}{r}{$-5851.30$} & \multicolumn{1}{r}{$1.42\%$} & $50.44$ & \multicolumn{1}{r}{$-5852.15$} & \multicolumn{1}{r}{$1.43\%$} & $45.80$ \\
& & & 1-4-multi-3000-4 & $-5783.57$ & \multicolumn{1}{r}{$-5855.41$} & \multicolumn{1}{r}{$1.24\%$} & $853.19$ & \multicolumn{1}{r}{$-5363.52$} & \multicolumn{1}{r}{$-7.26\%$} & $1120.67$ & \multicolumn{1}{r}{$-5855.88$} & \multicolumn{1}{r}{$1.25\%$} & $42.42$ \\
& & & 1-4-multi-3000-5 & $-5770.51$ & \multicolumn{1}{r}{$-5847.71$} & \multicolumn{1}{r}{$1.34\%$} & $759.02$ & \multicolumn{1}{r}{$-5847.90$} & \multicolumn{1}{r}{$1.34\%$} & $48.89$ & \multicolumn{1}{r}{$-5848.12$} & \multicolumn{1}{r}{$1.34\%$} & $79.78$ \\ \cline{2-14}
& \multirow{5}{*}{$39$} & \multirow{5}{*}{$5$} & 1-6-multi-3000-1 & $-9847.60$ & \multicolumn{1}{r}{$-4715.17$} & \multicolumn{1}{r}{$-52.12\%$} & $1050.76$ & \multicolumn{1}{r}{$-5552.73$} & \multicolumn{1}{r}{$-43.61\%$} & $1958.64$ & \multicolumn{1}{r}{$-9963.15$} & \multicolumn{1}{r}{$1.17\%$} & $482.06$ \\
& & & 1-6-multi-3000-2 & $-9866.40$ & \multicolumn{1}{r}{$-9986.35$} & \multicolumn{1}{r}{$1.22\%$} & $110.18$ & \multicolumn{1}{r}{$-9985.54$} & \multicolumn{1}{r}{$1.21\%$} & $756.33$ & \multicolumn{1}{r}{---} & \multicolumn{1}{r}{---} & $3600.00$ \\
& & & 1-6-multi-3000-3 & $-9871.75$ & \multicolumn{1}{r}{$-8897.43$} & \multicolumn{1}{r}{$-9.87\%$} & $275.73$ & \multicolumn{1}{r}{$-3778.40$} & \multicolumn{1}{r}{$-61.73\%$} & $294.53$ & \multicolumn{1}{r}{$-3453.30$} & \multicolumn{1}{r}{$-65.02\%$} & $339.25$ \\
& & & 1-6-multi-3000-4 & $-9883.64$ & \multicolumn{1}{r}{$-10008.02$} & \multicolumn{1}{r}{$1.26\%$} & $126.00$ & \multicolumn{1}{r}{$-6258.29$} & \multicolumn{1}{r}{$-36.68\%$} & $1639.07$ & \multicolumn{1}{r}{---} & \multicolumn{1}{r}{---} & $3600.00$ \\
& & & 1-6-multi-3000-5 & $-9852.74$ & \multicolumn{1}{r}{$-9984.75$} & \multicolumn{1}{r}{$1.34\%$} & $136.82$ & \multicolumn{1}{r}{$-5755.42$} & \multicolumn{1}{r}{$-41.59\%$} & $970.66$ & \multicolumn{1}{r}{---} & \multicolumn{1}{r}{---} & $3600.00$ \\ \cline{2-14}
& \multirow{5}{*}{$50$} & \multirow{5}{*}{$5$} & 1-7-multi-3000-1 & $-15778.48$ & \multicolumn{1}{r}{$-15902.91$} & \multicolumn{1}{r}{$0.79\%$} & $279.86$ & \multicolumn{1}{r}{$-15902.83$} & \multicolumn{1}{r}{$0.79\%$} & $82.59$ & \multicolumn{1}{r}{---} & \multicolumn{1}{r}{---} & $3600.00$ \\
& & & 1-7-multi-3000-2 & $-15750.66$ & \multicolumn{1}{r}{$-15886.40$} & \multicolumn{1}{r}{$0.86\%$} & $155.41$ & \multicolumn{1}{r}{$-15878.77$} & \multicolumn{1}{r}{$0.81\%$} & $225.76$ & \multicolumn{1}{r}{---} & \multicolumn{1}{r}{---} & $3600.00$ \\
& & & 1-7-multi-3000-3 & $-15782.39$ & \multicolumn{1}{r}{$-10966.13$} & \multicolumn{1}{r}{$-30.52\%$} & $344.26$ & \multicolumn{1}{r}{$-10540.66$} & \multicolumn{1}{r}{$-33.21\%$} & $660.41$ & \multicolumn{1}{r}{$-10191.29$} & \multicolumn{1}{r}{$-35.43\%$} & $2740.78$ \\
& & & 1-7-multi-3000-4 & $-15774.55$ & \multicolumn{1}{r}{$-15901.23$} & \multicolumn{1}{r}{$0.80\%$} & $160.82$ & \multicolumn{1}{r}{$-15913.76$} & \multicolumn{1}{r}{$0.88\%$} & $316.42$ & \multicolumn{1}{r}{$-15901.05$} & \multicolumn{1}{r}{$0.80\%$} & $408.35$ \\
& & & 1-7-multi-3000-5 & $-15744.71$ & \multicolumn{1}{r}{$-15881.81$} & \multicolumn{1}{r}{$0.87\%$} & $72.37$ & \multicolumn{1}{r}{$-15873.20$} & \multicolumn{1}{r}{$0.82\%$} & $101.42$ & \multicolumn{1}{r}{$-15874.82$} & \multicolumn{1}{r}{$0.83\%$} & $58.05$ \\ \hline
\multicolumn{6}{r}{Average}    & \multicolumn{1}{r}{$\mathbf{-7.66\%}$}      &                           & \multicolumn{1}{r}{}          & \multicolumn{1}{r}{$\mathbf{-18.77\%}$}      &                           & \multicolumn{1}{r}{}          & \multicolumn{1}{r}{$\mathbf{(-12.45\%)}$}     &                           \\ \hline
\multirow{15}{*}{$0.100333$} & \multirow{5}{*}{$20$} & \multirow{5}{*}{$10$} & 1-4-multi-3000-1 & $-5831.68$ & \multicolumn{1}{r}{$-5910.98$} & \multicolumn{1}{r}{$1.36\%$} & $960.78$ & \multicolumn{1}{r}{$-5907.93$} & \multicolumn{1}{r}{$1.31\%$} & $105.92$ & \multicolumn{1}{r}{$-5908.13$} & \multicolumn{1}{r}{$1.31\%$} & $65.95$ \\
& & & 1-4-multi-3000-2 & $-5828.17$ & \multicolumn{1}{r}{$-5912.97$} & \multicolumn{1}{r}{$1.46\%$} & $53.91$ & \multicolumn{1}{r}{$-5913.35$} & \multicolumn{1}{r}{$1.46\%$} & $38.30$ & \multicolumn{1}{r}{$-5913.26$} & \multicolumn{1}{r}{$1.46\%$} & $82.27$ \\
& & & 1-4-multi-3000-3 & $-5821.63$ & \multicolumn{1}{r}{$-4277.83$} & \multicolumn{1}{r}{$-26.52\%$} & $1553.48$ & \multicolumn{1}{r}{$-4103.27$} & \multicolumn{1}{r}{$-29.52\%$} & $1206.77$ & \multicolumn{1}{r}{$-5905.92$} & \multicolumn{1}{r}{$1.45\%$} & $57.44$ \\
& & & 1-4-multi-3000-4 & $-5829.83$ & \multicolumn{1}{r}{$-5909.92$} & \multicolumn{1}{r}{$1.37\%$} & $42.35$ & \multicolumn{1}{r}{$-5911.07$} & \multicolumn{1}{r}{$1.39\%$} & $50.75$ & \multicolumn{1}{r}{$-5911.35$} & \multicolumn{1}{r}{$1.40\%$} & $54.37$ \\
& & & 1-4-multi-3000-5 & $-5819.61$ & \multicolumn{1}{r}{$-5894.85$} & \multicolumn{1}{r}{$1.29\%$} & $51.62$ & \multicolumn{1}{r}{$-5895.32$} & \multicolumn{1}{r}{$1.30\%$} & $60.99$ & \multicolumn{1}{r}{$-4112.11$} & \multicolumn{1}{r}{$-29.34\%$} & $1120.31$ \\ \cline{2-14}
& \multirow{5}{*}{$39$} & \multirow{5}{*}{$5$} & 1-6-multi-3000-1 & $-9925.97$ & \multicolumn{1}{r}{$-10061.82$} & \multicolumn{1}{r}{$1.37\%$} & $220.22$ & \multicolumn{1}{r}{$-10062.52$} & \multicolumn{1}{r}{$1.38\%$} & $34.90$ & \multicolumn{1}{r}{$-10062.83$} & \multicolumn{1}{r}{$1.38\%$} & $46.56$ \\
& & & 1-6-multi-3000-2 & $-9945.20$ & \multicolumn{1}{r}{$-10086.21$} & \multicolumn{1}{r}{$1.42\%$} & $457.22$ & \multicolumn{1}{r}{$-10084.84$} & \multicolumn{1}{r}{$1.40\%$} & $132.32$ & \multicolumn{1}{r}{---} & \multicolumn{1}{r}{---} & $3600.00$ \\
& & & 1-6-multi-3000-3 & $-9956.57$ & \multicolumn{1}{r}{$-10107.55$} & \multicolumn{1}{r}{$1.52\%$} & $88.06$ & \multicolumn{1}{r}{$-10106.39$} & \multicolumn{1}{r}{$1.50\%$} & $1129.46$ & \multicolumn{1}{r}{---} & \multicolumn{1}{r}{---} & $3600.00$ \\
& & & 1-6-multi-3000-4 & $-9970.29$ & \multicolumn{1}{r}{$-10121.97$} & \multicolumn{1}{r}{$1.52\%$} & $65.15$ & \multicolumn{1}{r}{$-10119.67$} & \multicolumn{1}{r}{$1.50\%$} & $1103.21$ & \multicolumn{1}{r}{$-10120.66$} & \multicolumn{1}{r}{$1.51\%$} & $185.91$ \\
& & & 1-6-multi-3000-5 & $-9937.59$ & \multicolumn{1}{r}{$-6681.68$} & \multicolumn{1}{r}{$-32.76\%$} & $2564.60$ & \multicolumn{1}{r}{$-10083.03$} & \multicolumn{1}{r}{$1.46\%$} & $459.03$ & \multicolumn{1}{r}{---} & \multicolumn{1}{r}{---} & $3600.00$ \\ \cline{2-14}
& \multirow{5}{*}{$50$} & \multirow{5}{*}{$5$} & 1-7-multi-3000-1 & $-15863.65$ & \multicolumn{1}{r}{$-16019.40$} & \multicolumn{1}{r}{$0.98\%$} & $90.13$ & \multicolumn{1}{r}{$-16017.69$} & \multicolumn{1}{r}{$0.97\%$} & $1704.50$ & \multicolumn{1}{r}{$-16016.49$} & \multicolumn{1}{r}{$0.96\%$} & $176.73$ \\
& & & 1-7-multi-3000-2 & $-15838.68$ & \multicolumn{1}{r}{$-10946.83$} & \multicolumn{1}{r}{$-30.89\%$} & $2114.74$ & \multicolumn{1}{r}{---} & \multicolumn{1}{r}{---} & $3600.00$ & \multicolumn{1}{r}{$-16005.89$} & \multicolumn{1}{r}{$1.06\%$} & $1223.53$ \\
& & & 1-7-multi-3000-3 & $-15864.52$ & \multicolumn{1}{r}{$-12869.22$} & \multicolumn{1}{r}{$-18.88\%$} & $2280.27$ & \multicolumn{1}{r}{$-16002.81$} & \multicolumn{1}{r}{$0.87\%$} & $72.19$ & \multicolumn{1}{r}{$-16003.88$} & \multicolumn{1}{r}{$0.88\%$} & $757.04$ \\
& & & 1-7-multi-3000-4 & $-15860.93$ & \multicolumn{1}{r}{$-16019.43$} & \multicolumn{1}{r}{$1.00\%$} & $86.14$ & \multicolumn{1}{r}{$-16019.11$} & \multicolumn{1}{r}{$1.00\%$} & $135.34$ & \multicolumn{1}{r}{$-16012.86$} & \multicolumn{1}{r}{$0.96\%$} & $150.48$ \\
& & & 1-7-multi-3000-5 & $-15833.63$ & \multicolumn{1}{r}{$-15992.87$} & \multicolumn{1}{r}{$1.01\%$} & $126.80$ & \multicolumn{1}{r}{---} & \multicolumn{1}{r}{---} & $3600.00$ & \multicolumn{1}{r}{---} & \multicolumn{1}{r}{---} & $3600.00$ \\ \hline
\multicolumn{6}{r}{Average}    & \multicolumn{1}{r}{$\mathbf{-6.32\%}$}      &                           & \multicolumn{1}{r}{}          & \multicolumn{1}{r}{$\mathbf{(-1.07\%)}$}      &                           & \multicolumn{1}{r}{}          & \multicolumn{1}{r}{$\mathbf{(-1.54\%)}$}     &                           \\ \hline
\multirow{15}{*}{$0.200333$} & \multirow{5}{*}{$20$} & \multirow{5}{*}{$10$} & 1-4-multi-3000-1 & $-5882.23$ & \multicolumn{1}{r}{$-5980.90$} & \multicolumn{1}{r}{$1.68\%$} & $60.33$ & \multicolumn{1}{r}{$-5981.29$} & \multicolumn{1}{r}{$1.68\%$} & $92.64$ & \multicolumn{1}{r}{$-5982.58$} & \multicolumn{1}{r}{$1.71\%$} & $79.71$ \\
& & & 1-4-multi-3000-2 & $-5882.54$ & \multicolumn{1}{r}{$-5975.39$} & \multicolumn{1}{r}{$1.58\%$} & $118.26$ & \multicolumn{1}{r}{$-5944.34$} & \multicolumn{1}{r}{$1.05\%$} & $1335.93$ & \multicolumn{1}{r}{$-3775.14$} & \multicolumn{1}{r}{$-35.82\%$} & $1254.68$ \\
& & & 1-4-multi-3000-3 & $-5878.65$ & \multicolumn{1}{r}{---} & \multicolumn{1}{r}{---} & $3600.00$ & \multicolumn{1}{r}{$-5981.66$} & \multicolumn{1}{r}{$1.75\%$} & $79.12$ & \multicolumn{1}{r}{$-3953.03$} & \multicolumn{1}{r}{$-32.76\%$} & $1008.98$ \\
& & & 1-4-multi-3000-4 & $-5885.45$ & \multicolumn{1}{r}{$-5985.57$} & \multicolumn{1}{r}{$1.70\%$} & $80.63$ & \multicolumn{1}{r}{$-5985.85$} & \multicolumn{1}{r}{$1.71\%$} & $72.83$ & \multicolumn{1}{r}{$-3801.25$} & \multicolumn{1}{r}{$-35.41\%$} & $1060.49$ \\
& & & 1-4-multi-3000-5 & $-5873.65$ & \multicolumn{1}{r}{$-5969.43$} & \multicolumn{1}{r}{$1.63\%$} & $96.34$ & \multicolumn{1}{r}{$-5969.97$} & \multicolumn{1}{r}{$1.64\%$} & $70.32$ & \multicolumn{1}{r}{$-5970.48$} & \multicolumn{1}{r}{$1.65\%$} & $90.51$ \\ \cline{2-14}
& \multirow{5}{*}{$39$} & \multirow{5}{*}{$5$} & 1-6-multi-3000-1 & $-10021.39$ & \multicolumn{1}{r}{$-10193.10$} & \multicolumn{1}{r}{$1.71\%$} & $72.18$ & \multicolumn{1}{r}{$-10192.75$} & \multicolumn{1}{r}{$1.71\%$} & $738.18$ & \multicolumn{1}{r}{$-10194.72$} & \multicolumn{1}{r}{$1.73\%$} & $224.50$ \\
& & & 1-6-multi-3000-2 & $-10041.39$ & \multicolumn{1}{r}{$-10214.84$} & \multicolumn{1}{r}{$1.73\%$} & $86.70$ & \multicolumn{1}{r}{$-10212.85$} & \multicolumn{1}{r}{$1.71\%$} & $215.10$ & \multicolumn{1}{r}{$-10212.41$} & \multicolumn{1}{r}{$1.70\%$} & $444.14$ \\
& & & 1-6-multi-3000-3 & $-10060.69$ & \multicolumn{1}{r}{$-10245.91$} & \multicolumn{1}{r}{$1.84\%$} & $77.05$ & \multicolumn{1}{r}{---} & \multicolumn{1}{r}{---} & $3600.00$ & \multicolumn{1}{r}{$-10246.20$} & \multicolumn{1}{r}{$1.84\%$} & $2295.06$ \\
& & & 1-6-multi-3000-4 & $-10077.44$ & \multicolumn{1}{r}{---} & \multicolumn{1}{r}{---} & $3600.00$ & \multicolumn{1}{r}{$-10274.30$} & \multicolumn{1}{r}{$1.95\%$} & $199.66$ & \multicolumn{1}{r}{$-10270.80$} & \multicolumn{1}{r}{$1.92\%$} & $380.29$ \\
& & & 1-6-multi-3000-5 & $-10037.88$ & \multicolumn{1}{r}{$-10229.68$} & \multicolumn{1}{r}{$1.91\%$} & $944.48$ & \multicolumn{1}{r}{$-10230.54$} & \multicolumn{1}{r}{$1.92\%$} & $144.37$ & \multicolumn{1}{r}{$-10228.89$} & \multicolumn{1}{r}{$1.90\%$} & $209.22$ \\ \cline{2-14}
& \multirow{5}{*}{$50$} & \multirow{5}{*}{$5$} & 1-7-multi-3000-1 & $-15971.22$ & \multicolumn{1}{r}{$-16161.54$} & \multicolumn{1}{r}{$1.19\%$} & $80.40$ & \multicolumn{1}{r}{$-16160.78$} & \multicolumn{1}{r}{$1.19\%$} & $110.50$ & \multicolumn{1}{r}{$-16160.89$} & \multicolumn{1}{r}{$1.19\%$} & $264.81$ \\
& & & 1-7-multi-3000-2 & $-15953.23$ & \multicolumn{1}{r}{---} & \multicolumn{1}{r}{---} & $3600.00$ & \multicolumn{1}{r}{$-16154.10$} & \multicolumn{1}{r}{$1.26\%$} & $115.48$ & \multicolumn{1}{r}{$-16155.65$} & \multicolumn{1}{r}{$1.27\%$} & $246.37$ \\
& & & 1-7-multi-3000-3 & $-15966.21$ & \multicolumn{1}{r}{$-16161.73$} & \multicolumn{1}{r}{$1.22\%$} & $110.24$ & \multicolumn{1}{r}{$-16160.10$} & \multicolumn{1}{r}{$1.21\%$} & $128.82$ & \multicolumn{1}{r}{$-16160.41$} & \multicolumn{1}{r}{$1.22\%$} & $264.63$ \\
& & & 1-7-multi-3000-4 & $-15973.34$ & \multicolumn{1}{r}{$-16166.21$} & \multicolumn{1}{r}{$1.21\%$} & $77.53$ & \multicolumn{1}{r}{---} & \multicolumn{1}{r}{---} & $3600.00$ & \multicolumn{1}{r}{---} & \multicolumn{1}{r}{---} & $3600.00$ \\
& & & 1-7-multi-3000-5 & $-15948.75$ & \multicolumn{1}{r}{$-16150.54$} & \multicolumn{1}{r}{$1.27\%$} & $126.65$ & \multicolumn{1}{r}{$-16149.75$} & \multicolumn{1}{r}{$1.26\%$} & $102.94$ & \multicolumn{1}{r}{---} & \multicolumn{1}{r}{---} & $3600.00$ \\ \hline
\multicolumn{6}{r}{Average}    & \multicolumn{1}{r}{$\mathbf{(1.56\%)}$}      &                           & \multicolumn{1}{r}{}          & \multicolumn{1}{r}{$\mathbf{(1.54\%)}$}      &                           & \multicolumn{1}{r}{}          & \multicolumn{1}{r}{$\mathbf{(-6.76\%)}$}     &                           \\ \hline
\multirow{15}{*}{$0.300333$} & \multirow{5}{*}{$20$} & \multirow{5}{*}{$10$} & 1-4-multi-3000-1 & $-5918.46$ & \multicolumn{1}{r}{$-5529.07$} & \multicolumn{1}{r}{$-6.58\%$} & $1109.47$ & \multicolumn{1}{r}{$-5967.94$} & \multicolumn{1}{r}{$0.84\%$} & $1219.53$ & \multicolumn{1}{r}{$-6035.95$} & \multicolumn{1}{r}{$1.99\%$} & $115.28$ \\
& & & 1-4-multi-3000-2 & $-5918.62$ & \multicolumn{1}{r}{$-6031.59$} & \multicolumn{1}{r}{$1.91\%$} & $83.70$ & \multicolumn{1}{r}{$-6031.64$} & \multicolumn{1}{r}{$1.91\%$} & $66.21$ & \multicolumn{1}{r}{$-6032.01$} & \multicolumn{1}{r}{$1.92\%$} & $160.50$ \\
& & & 1-4-multi-3000-3 & $-5917.11$ & \multicolumn{1}{r}{$-6032.30$} & \multicolumn{1}{r}{$1.95\%$} & $68.59$ & \multicolumn{1}{r}{$-6032.78$} & \multicolumn{1}{r}{$1.95\%$} & $540.16$ & \multicolumn{1}{r}{$-4373.25$} & \multicolumn{1}{r}{$-26.09\%$} & $1084.91$ \\
& & & 1-4-multi-3000-4 & $-5923.33$ & \multicolumn{1}{r}{$-6037.02$} & \multicolumn{1}{r}{$1.92\%$} & $299.47$ & \multicolumn{1}{r}{$-6037.10$} & \multicolumn{1}{r}{$1.92\%$} & $134.30$ & \multicolumn{1}{r}{$-6036.37$} & \multicolumn{1}{r}{$1.91\%$} & $103.80$ \\
& & & 1-4-multi-3000-5 & $-5911.10$ & \multicolumn{1}{r}{$-4329.24$} & \multicolumn{1}{r}{$-26.76\%$} & $1755.72$ & \multicolumn{1}{r}{$-6030.15$} & \multicolumn{1}{r}{$2.01\%$} & $139.35$ & \multicolumn{1}{r}{$-3832.07$} & \multicolumn{1}{r}{$-35.17\%$} & $1209.61$ \\ \cline{2-14}
& \multirow{5}{*}{$39$} & \multirow{5}{*}{$5$} & 1-6-multi-3000-1 & $-10091.29$ & \multicolumn{1}{r}{$-10301.80$} & \multicolumn{1}{r}{$2.09\%$} & $185.77$ & \multicolumn{1}{r}{$-10302.85$} & \multicolumn{1}{r}{$2.10\%$} & $1137.36$ & \multicolumn{1}{r}{$-10303.76$} & \multicolumn{1}{r}{$2.11\%$} & $290.70$ \\
& & & 1-6-multi-3000-2 & $-10110.44$ & \multicolumn{1}{r}{$-10319.13$} & \multicolumn{1}{r}{$2.06\%$} & $78.86$ & \multicolumn{1}{r}{$-10320.36$} & \multicolumn{1}{r}{$2.08\%$} & $108.68$ & \multicolumn{1}{r}{$-10320.34$} & \multicolumn{1}{r}{$2.08\%$} & $299.29$ \\
& & & 1-6-multi-3000-3 & $-10136.15$ & \multicolumn{1}{r}{$-10350.59$} & \multicolumn{1}{r}{$2.12\%$} & $66.52$ & \multicolumn{1}{r}{$-10351.18$} & \multicolumn{1}{r}{$2.12\%$} & $107.02$ & \multicolumn{1}{r}{$-10352.26$} & \multicolumn{1}{r}{$2.13\%$} & $248.15$ \\
& & & 1-6-multi-3000-4 & $-10154.42$ & \multicolumn{1}{r}{$-10393.37$} & \multicolumn{1}{r}{$2.35\%$} & $79.72$ & \multicolumn{1}{r}{$-10395.64$} & \multicolumn{1}{r}{$2.38\%$} & $105.96$ & \multicolumn{1}{r}{$-10394.16$} & \multicolumn{1}{r}{$2.36\%$} & $193.27$ \\
& & & 1-6-multi-3000-5 & $-10112.49$ & \multicolumn{1}{r}{$-10328.75$} & \multicolumn{1}{r}{$2.14\%$} & $103.98$ & \multicolumn{1}{r}{$-10328.43$} & \multicolumn{1}{r}{$2.14\%$} & $112.36$ & \multicolumn{1}{r}{---} & \multicolumn{1}{r}{---} & $3600.00$ \\ \cline{2-14}
& \multirow{5}{*}{$50$} & \multirow{5}{*}{$5$} & 1-7-multi-3000-1 & $-16046.86$ & \multicolumn{1}{r}{$-16272.00$} & \multicolumn{1}{r}{$1.40\%$} & $134.34$ & \multicolumn{1}{r}{---} & \multicolumn{1}{r}{---} & $3600.00$ & \multicolumn{1}{r}{---} & \multicolumn{1}{r}{---} & $3600.00$ \\
& & & 1-7-multi-3000-2 & $-16035.80$ & \multicolumn{1}{r}{$-16280.56$} & \multicolumn{1}{r}{$1.53\%$} & $96.06$ & \multicolumn{1}{r}{$-16284.22$} & \multicolumn{1}{r}{$1.55\%$} & $126.57$ & \multicolumn{1}{r}{$-16285.86$} & \multicolumn{1}{r}{$1.56\%$} & $279.63$ \\
& & & 1-7-multi-3000-3 & $-16043.21$ & \multicolumn{1}{r}{$-16271.90$} & \multicolumn{1}{r}{$1.43\%$} & $105.23$ & \multicolumn{1}{r}{$-16272.40$} & \multicolumn{1}{r}{$1.43\%$} & $160.19$ & \multicolumn{1}{r}{$-16275.51$} & \multicolumn{1}{r}{$1.45\%$} & $880.99$ \\
& & & 1-7-multi-3000-4 & $-16050.82$ & \multicolumn{1}{r}{$-16270.86$} & \multicolumn{1}{r}{$1.37\%$} & $100.12$ & \multicolumn{1}{r}{$-16271.18$} & \multicolumn{1}{r}{$1.37\%$} & $96.51$ & \multicolumn{1}{r}{$-16272.31$} & \multicolumn{1}{r}{$1.38\%$} & $212.77$ \\
& & & 1-7-multi-3000-5 & $-16030.74$ & \multicolumn{1}{r}{$-16266.31$} & \multicolumn{1}{r}{$1.47\%$} & $132.44$ & \multicolumn{1}{r}{$-15314.50$} & \multicolumn{1}{r}{$-4.47\%$} & $120.63$ & \multicolumn{1}{r}{---} & \multicolumn{1}{r}{---} & $3600.00$ \\ \hline
\multicolumn{6}{r}{Average}    & \multicolumn{1}{r}{$\mathbf{-0.64}\%$} &                           & \multicolumn{1}{r}{}          & \multicolumn{1}{r}{$(\mathbf{1.38}\%)$}      &                           & \multicolumn{1}{r}{}          & \multicolumn{1}{r}{$(\mathbf{-3.53}\%)$}     &                           \\ \hline
\end{tabular}
}
\end{table}

\end{document}